\documentclass[a4paper,11pt,reqno]{amsart}

\usepackage{xurl}
\usepackage{graphicx}
\usepackage{graphics}
\usepackage{relsize}
\usepackage{caption}
\usepackage{amsmath}
\usepackage{amssymb}
\usepackage{amsthm}
\usepackage{float}
\usepackage{pst-node}
\usepackage{tikz-cd} 
\usepackage[original]{imakeidx}
\makeindex[columns=2, title=Index]
\usepackage{hyperref}
\usepackage{cleveref}
\usepackage{needspace}
\usepackage{etoolbox}

\usepackage{geometry}
\author{Francisco Araújo}
\title{Sarnak's Program for Erd\H{o}s Sieves. Part I: Topological Dynamics and Light Tails}

\theoremstyle{plain}
\newtheorem{theorem}{Theorem}[section]

\newtheorem{lemma}[theorem]{Lemma}
\newtheorem{definition}[theorem]{Definition}
\newtheorem{proposition}[theorem]{Proposition}
\newtheorem{corollary}[theorem]{Corollary}
\newtheorem{conjecture}[theorem]{Conjecture}

\theoremstyle{remark}
\newtheorem{example}[theorem]{Example}
\newtheorem{remark}[theorem]{Remark}

\makeatletter

\pretocmd{\theorem}{\needspace{4\baselineskip}}{}{}
\pretocmd{\lemma}{\needspace{3\baselineskip}}{}{}
\pretocmd{\proposition}{\needspace{4\baselineskip}}{}{}
\pretocmd{\corollary}{\needspace{3\baselineskip}}{}{}
\pretocmd{\conjecture}{\needspace{3\baselineskip}}{}{}

\pretocmd{\definition}{\needspace{2\baselineskip}}{}{}

\makeatother

\DeclareMathOperator{\Gen}{Gen}
\DeclareMathOperator{\vol}{vol}
\DeclareMathOperator{\hpc}{h_{\text{pc}}}
\DeclareMathOperator{\htop}{h_{\text{top}}}

\newcommand{\q}{\mathbb{Q}}
\newcommand{\z}{\mathbb{Z}}

\newcommand{\cb}{\mathcal{B}}
\newcommand{\cbr}{\mathcal{B}_R}

\newcommand{\co}{\mathcal{O}}

\newcommand{\cf}{\mathcal{F}}

\newcommand{\fr}{\mathcal{F}_R}

\newcommand{\frf}{\mathcal{F}_{R^f}}

\newcommand{\frp}{\mathcal{F}_{R'}}

\newcommand{\ga}{\mathfrak{a}}
\newcommand{\gb}{\mathfrak{b}}
\newcommand{\gc}{\mathfrak{c}}
\newcommand{\gp}{\mathfrak{p}}

\newcommand{\gr}{G_{R}}

\newcommand{\mmp}{\mathbb{P}}

\newcommand*{\one}{\text{\usefont{U}{bbold}{m}{n}1}}

\begin{document}
	
	\begin{abstract}
		This paper is the first part of a two-part article where we generalize Sarnak's program to sets where we remove congruence classes modulo some infinite set $\mathcal{B}$ of ideals of an étale $\mathbb{Q}-$algebra $K$, which we denote by Erdős sieves. We define some light tail conditions on a sieve $R$, and show how these are related to the genericity under the Mirsky measure of the set of $R-$free numbers, which are the algebraic integers of $K$ not contained in any of the congruence classes in $R$. We also show that Erdős $\mathcal{B}-$free systems in any étale $\mathbb{Q}-$algebra satisfy these light tail conditions, so our results generalize Sarnak's program to Erdős $\mathcal{B}-$free systems over any étale $\mathbb{Q}-$algebra.
	\end{abstract}
	
\maketitle

	\section{Introduction}
	
	This is the first part of a series of two papers generalizing Sarnak's program for Erdős sieves. The main motivation for the study of these is to better understand sets of algebraic integers obtained by removing congruence classes modulo an infinite family of ideals $\cb$ from a certain ambient space. Such sets have been studied over $\z$ for the past century, mainly in the case where all congruence classes are $0$. The canonical example of such sets would be the squarefree integers. To obtain these, one takes the set $\cb = \{p^2\z: p \text{ prime}\}$, and then takes the integers that are not congruent to $0$ modulo $p^2$ for any prime $p$. But many sets can be obtained in such a manner. For example, the set of primes corresponds to positive integers that are not congruent to $0$ modulo $pq$ for any primes $p$ and $q$. One of the most basic questions about such sets is related to their density. In \cite{Besicovitch}, Besicovitch showed that there are sets $\cb$ of ideals of $\z$ such that the set of integers not in $b\z$ for every $b\z \in \cb$ does not have a well defined density. Yet, Erdős and Davenport showed in \cite{Davenport} that such sets always have a logarithmic density. 
	
	Going beyond $\z$, sets of this form have also been studied over $\z^n$ and over the ring of integers $\co_K$ of some number field $K$. For example, the set of $l-$free numbers in a number field, can be obtained by removing from $\co_K$ all integers which are not congruent to $0$ modulo $\gb$ for any ideal $\gb$ contained in the set $\cb = \{\gp^l: \gp \text{ prime of }\co_K\}$. Also the set of \index{Visible Lattice Points}\textit{visible lattice points}, which is the set of pairs $(x,y) \in \z^2$ where $x$ and $y$ are coprime, can be obtained from $\z^2$ by removing every element of the ideals in $\cb = \{p\z \times p\z: p \text{ prime}\}$. In order to be able to deal with all such examples, we will work with sets $\cb$ of ideals in the ring of integers $\co_K$, where $K$ is taken to be an étale $\q-$algebra (i.e. the product of finitely many number fields, see \Cref{def: Etale Algebra}).
	
	We are interested in what happens when one chooses a set of  ideals $\cb$, and removes congruence classes modulo $\gb \in \cb$ distinct from only $0$.  We formalize this with the notion of \textit{sieves}. By a sieve over an étale $\q-$algebra $K$ we mean a pair $(\cb_R,(R_\gb)_{\gb \in \cb_R})$, where $\cb_R$ are a set of pairwise coprime ideals in $\co_K$, and each $R_\gb$ is a set of congruence classes mod $\gb$. In this work we will only deal with the case where $\cb_R$ is a set of pairwise coprime ideals.
	
	Having defined a sieve $R$, we want to study the set of $R-$free numbers $$\fr := \co_K \setminus\left(\bigcup_{\gb \in \cb_R} R_\gb\right).$$ When a sieve is of the form $R_\gb = \gb$ for all $\gb \in \cb_R$, we will refer to it as a $\cb-$free system. We have just described how in the case of $\cb-$free systems (especially over $\q$), quite a lot has been done in regards to studying the density of $\fr$. In contrast, other types of sets that can be realized as the $R-$free numbers for some sieve $R$ have been far less studied.
	
	Yet, from a number theoretic point of view, there are many interesting sets that are of this type. An example, which we investigate in part II, would be the set of squarefree values of polynomials. Let $f\in \z[X]$ be an irreducible polynomial. Studying the set of values $x$ for which $f(x)$ is squarefree is a problem that has been investigated by number theorists for almost one century (see for example \cite{Brwoning},\cite{estermann},\cite{Filaseta},\cite{granville} or \cite{Hooley}). Notice how, if we let $R^f$ be the sieve with congruence classes modulo $p^2$ given by the set
	\begin{equation}
		\label{eq: Rfp intro}
		R^f_p = \{x\in \z/p^2\z: f(x) \equiv 0 \mod p^2\} + p^2\z,
	\end{equation} we have that $\frf$ is exactly the set of values $x$ for which $f(x)$ is squarefree.
	
	If $R$ is a $\cb-$free system with $\cb_R =\cb$, we will write $\cf_{\cb}$ for $\fr$. In the last twenty five years, there has been extensive study of shift spaces related to these sets (see \cite{Baake} or \cite{Pleasants} for some early examples). There are two related shift spaces that one tends to study. In the case of $\cb-$free systems they usually agree, but this does not hold for general sieves. Assuming that $\cb$ is a set of ideals of $\co_K$, both spaces live inside $\{0,1\}^{\co_k}$ (which we identify with the power set of $\co_K$) where there is a shift action $S_a(A) = -a+A$ for any $a \in \co_K$ and set $A\subset \co_K$. The first space, which we denote by $X_\cb$, is the closure of the orbit of $\cf_\cb$ under $S$, inside of $\{0,1\}^{\co_k}$. The second is the collection of all $\cb-$admissible sets, which are sets $A\subset \co_K$ such that $A+\gb \neq \co_K$ for all $\gb \in \cb$. We will denote this set by $\Omega_{\cb}$.
	
	In the seminal paper \cite{Sarnak}, Sarnak conjectured\index{Sarnak's ! Conjecture} that, denoting by $\mu$ the Möbius function, we have that if $(X,T)$ is a topological dynamical system with entropy $0$ we have, for all $x \in X$ and $f\in C(X)$, that $$\sum_{n \leq N} \mu(n)f(T^nx) = o(N).$$ Note that this generalizes many famous results in number theory, for example in the case where $X$ is a single point, this is just the Prime Number Theorem (see \Cref{thm: Prime Ideal Theorem}). In order to tackle this problem, Sarnak started by considering the shift spaces associated to the squarefree numbers (note that $\mu^2$ is the characteristic function of the squarefree numbers). Letting $\cb = \{p^2\z: p \text{ prime}\}$, Sarnak announced the following results.
	
	\newpage

	\begin{enumerate}
		\item There is a measure $\nu_\cb$ (the Mirsky measure) for which $\cf_\cb$ is a generic point (see \Cref{def: generic point}) of the system $(X_\cb,S,\nu_\cb)$.
		\item The topological entropy of $(X_\cb,S)$ is $\frac{6}{\pi^2}$.
		\item $X_\cb = \Omega_{\cb}$.
		\item The system $(X_\cb,S)$ is proximal (see \Cref{eq: def of proximality}) and its unique minimal system is $\{\emptyset\}$.
		\item There is a nontrivial joining  of the systems $(X_\cb,S)$ and $(G,T)$ (see \Cref{def: Joining}), where $G$ is the group $\prod_p \z/p^2\z$ and $T$ is the rotation on $G$ defined by $T(g)_p := g_p +1$.
	\end{enumerate}
	
	These five points together form what is referred to in the literature as \textit{Sarnak's program}. Sarnak didn't publish detailed proofs of these results, but his notes inspired many authors that further generalized these results to many different contexts. We now give a short overview of some of these developments.
	
	In \cite{celarosi2}, Cellarosi and Sinai investigated the squarefree flow $(X_{\cb},S,\nu_\cb)$, and showed that it is isomorphic to $(G,T,\mmp)$, where $\mmp$ denotes the Haar measure of $G$. Peckner showed in \cite{Peckner} that the system $(X_{\cb},S)$ has a unique measure of maximum entropy. In \cite{Abdalaoui}, El Abdalaoui, Lemańczyk and De La Rue generalized points $(1)$, $(2)$ and $(3)$ of Sarnak's program for Erdős $\cb-$free systems over $\q$. These are $\cb-$free systems where the elements of $\cb$ are pairwise coprime and $\sum_{b\z\in \cb}1/b$ is finite. They also showed that there is an isomorphism between the system $(X_\cb,S,\nu_\cb)$ and a rotation of the group $\prod_{b \in \cb} \z/b\z$. The paper \cite{Dymek} of Dymek,  Kasjan,  Kułaga-Przymus, and Lemańczyk provides an extensive treatment of all $\cb-$free systems over $\q$ (including the case where the elements of $\cb$ are not pairwise coprime).

	These results were then generalized for $\cb-$free systems over different étale $\q-$algebras $K$. In the case where $K = \q^n$, Pleasants and Huck in \cite{Pleasants} proved points $(1)$, $(2)$ and $(3)$ of Sarnak's program for $k-$free lattice points. This corresponds to the case where $\cb = \{p^k\z \times \dots \times p^k\z: p \text{ prime}\}$. Additionally, they showed that the system $(X_\cb,S,\nu_\cb)$ is isomorphic to a rotation of the group $G= \prod_p \z^n/(p^k\z \times \dots \times p^k\z)$. In the case of number fields, Cellarosi and Vinogradov proved a generalization of point (1) when $K$ is a number field with ring of integers $\co_K$ and $\cb = \{\gp^k: \gp \text{ prime ideal of } \co_K\}$ and showed that the system $(\Omega_\cb,S,\nu_\cb)$ is isomorphic to a rotation of a compact group. Finally, in \cite{Bartnicka}, the authors generalized all of Sarnak's program to the case where $K$ is a number field, and $\cb$ is a set of ideals of $\co_K$ that are pairwise coprime and such that $\sum_{\gb \in \cb}N(\gb)^{-1}$ is finite (where $N(\gb) = |\co_K/\gb|$ is the norm of $\gb$).

	The notion of a sieve in this context was introduced by Gundlach and Klüners in \cite{Fabian}, while generalizing the results of \cite{Baake 4}. For a sieve $R$, they define the space $X_R$, corresponding to the closure of the orbit of $\fr$ under $S$ inside of $\{0,1\}^{\co_K}$. They also defined the space $\Omega_{R}$ of admissible sets, which are those $A \subset \co_K$ such that $-A+ R_\gb \neq \co_K$ for all $\gb \in \cb_R$.  Their work focused mostly on studying the symmetry group of the system $(\Omega_{R},S)$, and dealt with no measure theoretical dynamics. Since sieves generalize $\cb-$free systems, the first objective of this thesis was to prove a version of Sarnak's program for general sieves. In attempting to do this, we had to define a number of properties for sieves. The most important of these, were the notions of an Erdős sieve, as well as of a sieve with   strong and weak light tails. 
	
	We say a sieve is \textit{Erdős}, if $$\sum_{\gb \in \cb_R} \frac{|R_\gb|}{N(\gb)} < \infty.$$ For any sieve, we are free to order $\cb_R = \{\gb_1,\gb_2, \dots\}$, in which case we write $R_i := R_{\gb_i}$. Given a Følner sequence $I_N$ (see \Cref{def: Følner sequence}), we say that an Erdős sieve $R$ has \textit{weak light tails} for $I_N$ if $$\lim_{L \rightarrow \infty} \overline{d}_I\left(\bigcup_{i>L}R_i \setminus \bigcup_{j \leq L} R_j\right) = 0,$$ where $\overline{d}_I$ is the upper density with respect to $I_N$ (see \Cref{def: density}). We also say that $R$ has \textit{strong light tails} if $$\lim_{L \rightarrow \infty} \overline{d}_I\left(\bigcup_{i>L}R_i\right) = 0.$$ Strong light tails implies weak light tails, but not the other way around. In \Cref{ex:Sieve with weak but not strong light tails} we provide a sieve that has weak light tails for $I_N = [0,N]$, but does not have strong light tails. In general, sieves with weak light tails have well behaved dynamical systems, but the property may not be stable (for example, removing just one congruence class from an Erdős sieve $R$ with weak light tails for some $I_N$ can produce a sieve without weak light tails for any Følner sequence).

	We now present the results obtained. To our knowledge, no work had been done on the systems $(X_R,S)$ and $(\Omega_{R},S)$ previous to  \cite{Fabian}, except for showing that when $R$ is a sieve over $\q$, the system $(\Omega_R,S)$ has a unique measure of maximum entropy for certain $R$, see Theorem 2.2.25 in \cite{Kulaga} (a more general result was also obtained in \cite{Dymek}).
	
		Between this paper and part II, we will provide a full generalization of Sarnak's Program for Erdős Sieves, together with some number theoretic applications coming from these results (which we present in part II). The objective of this paper is to present the basic tools with which we will work, as well as generalizing all points of Sarnak's program except for (3). In part II we will provide a generalization of point (3), show that the system $(\Omega_{R},S,\nu_R)$ is isomorphic to a minimal rotation of a compact group, and provide some applications of these results.
	
	This paper is divided as follows. In section 2 we present some basic results from both algebraic number theory and the theory of dynamical system that we will use. In section 3 we define sieves and some of their properties, together with the corresponding dynamical systems that we will be studying. The major result of this section  (see \Cref{thm:Fr is generic}) is the following generalization of point $(1)$ of Sarnak's program, which connects the system $(\Omega_R,S,\nu_R)$, where $\nu_R$ is the Mirsky measure in $\Omega_R$ (see \Cref{def: Mirsky measure}), the weak light tails property and the set $\fr$.
	
	\begin{theorem}
		\label{thm: intro fr generic}
		Let $R$ be an Erdős sieve. For a given Følner sequence $I_N$, the following are equivalent.
		
		\begin{enumerate}
			\item $\fr$ is a generic point of $(\Omega_{R},S,\nu_R)$ with respect to $I_N$.
			\item $R$ has weak light tails with respect to $I_N$.
			\item The set $\fr$ has a density with respect to $I_N$ given by $$d_I(\fr) =  \prod_{\gb \in \cbr}  \left(1-\frac{|R_\gb|}{N(\gb)}\right). $$
		\end{enumerate}

	\end{theorem}

 In section $4$ we generalize points $(2),(4)$ and $(5)$ of Sarnak's program for sieves. Finally in section $5$ we investigate the weak and strong light tails properties. The main result of this section is \Cref{thm: K etale algebra light tails for B_n finite} where we show that any Erdős $\cb-$free system over an étale $\q-$algebra has strong light tails for the Følner sequence $B_N$ (as defined in \Cref{eq: definition of BN}). Additionally, we investigate under which conditions does a sieve with weak light tails also have strong light tails.\\

	\textbf{Acknowledgments}\\
	
	The author would like to thank Jürgen Klüners, Joanna Kułaga-Przymus, Fabian Gundlach, Aurelia Dymek and Michael Baake for many helpful discussions and comments that greatly contributed for this work. This research was supported by the Deutsche Forschungsgemeinschaft (DFG, German Research Foundation) - Project-ID 491392403 - TRR 358 (project A2).

	\section{Preliminaries}
	
	Throughout, the natural numbers $\mathbb{N}$ will mean the integers greater than or equal to $1$. For the union of this set with $\{0\}$, we will write $\mathbb{N}_0$. The following is a well know result in analysis, that due to how often we will use throughout, we write down as a lemma (see Exercise 1.32 in \cite{Folland}).
	
	\begin{lemma}
		\label{lm: finite sum finite products}
		If $0 < a_i < 1$ are real numbers, then $$\sum_{i \geq 1} a_i < \infty \Leftrightarrow \prod_{i \geq 1} (1-a_i) > 0. $$
		
	\end{lemma}

	\subsection{Lattices in étale $\q-$algebras}  By a \index{Number Field}number field $K$ we mean a finite field extension of the rationals $\q$, such as the quadratic number fields $\q[\sqrt{d}]$, or the cyclotomic fields $\q[e^{2\pi i/n}]$. We say $K$ has degree $n$ if  looking at it as a vector space over $\q$, it has dimension $n$. For example, if $d$ is not a square, then $\q[\sqrt{d}]$ has degree 2, and $\q[e^{2\pi i/n}]$ always has degree $\varphi(n)$, where $\varphi$ is the Euler totient function. It is well known (see for example Theorem 51 in the Appendix B of \cite{Marcus}) that for every $K$ there is some $\alpha \in K$ such that $K = \q[\alpha]$.

	We denote by $\co_K$ the \index{Ring of Integers}ring of integers of $K$, that is, those $x \in K$ that are roots of a monic polynomial with coefficients in $\z$. As a ring, $\co_K$ is always a Dedekind domain, meaning that any non-zero ideal can be uniquely factorized as a product of non-zero prime ideals, which correspond to the ideals not contained in any other proper ideal. As an additive group, $\co_K$ is always isomorphic to $\z^n$ (see Theorem 2 in Section 2 of \cite{Marcus}), so we can think of it as a \index{Lattice}lattice. By a lattice in $\co_K$, we mean a free $\z-$module of rank $n$ in $\co_K$. Each non-zero ideal $\gb$ of $\co_K$ is a lattice, and we denote by the \index{Norm of an Ideal}\textit{norm} of an ideal the quantity $$N(\gb) := [\co_K:\gb] = |\co_K/\gb|,$$ which is always finite. For any two ideals $\ga$ and $\gb$, we have $N(\ga\gb) = N(\ga)N(\gb)$. 
	
	The same way we have the Prime Number Theorem over $\q$, we have the \index{Prime Number Theorem ! for Number Fields}Prime Ideal Theorem over any number field, as proven by Landau in the second part of \cite{Landau}.
	
	\begin{theorem}
		\label{thm: Prime Ideal Theorem}
		Let $K$ be a number field with ring of integers $\co_K$. Denoting by $\pi_K(X)$ the number of prime ideals $\gp$ of $\co_K$ such that $N(\gp) \leq X$, we have $$\lim_{X \rightarrow \infty} \frac{\pi_K(X)}{X/\log(X)} = 1.$$
	\end{theorem}

	Whenever $A$ and $B$ are two subsets of an abelian semigroup (like $\mathbb{N}$ or $\co_K$ for some number field $K$), we will write $$A+B := \{a+b: a\in A, b\in B\},$$ with $A + \emptyset = \emptyset$.  Given two ideals $\ga,\gb \subset \co_K$, we say they are coprime if $\ga + \gb = \co_K$. In this case, we will write $(\ga,\gb) = 1$. For ideals we have that $\ga \mid \gb$ is equivalent to $\gb \subset \ga$, so we have that the greatest common divisor of $\ga$ and $\gb$, which we write as $(\ga,\gb)$, equals $ \ga+\gb$. Given a set $\gb_1,\dots, \gb_k$ of ideals, we will denote by $\text{lcm}(\{\gb_1,\dots,\gb_k\})$ the greatest ideal that is contained in all of $\gb_1,\dots,\gb_k$, that is $$\text{lcm}(\{\gb_1,\dots,\gb_k\}) = \gb_1\cap \dots \cap \gb_k. $$ If $\ga$ and $\gb$ are coprime, this corresponds to $\ga \cap \gb = \ga\gb$. 
	
	For a number field $K$ of degree $n$, there are $n$ distinct embeddings of $K$ into $\mathbb{C}$. Writing $K = \q[\alpha]$, and letting $f$ be the minimal polynomial of $\alpha$, these correspond to the field homomorphisms determined by $\phi(\alpha) = \theta$, where $\theta$ is one of the $n$ distinct roots of $f$ in $\mathbb{C}$
	
	We then consider the Minkowski embedding\index{Minkowski Embedding} $\sigma: \co_K \rightarrow \mathbb{C}^n$ given by $$\sigma(x):= (\phi(x))_{\phi \in \text{Hom}_\q(K, \mathbb{C})}.$$
	We can define a (vector field) norm on $K$ by taking the norm inherited from the supremum norm in the Minkowski embedding, that is
	\begin{equation}
		\label{eq: minkowsky norm}
		\|x\| := \sup_{\phi \in \text{Hom}_\q(K,\mathbb{C})}|\phi(x)|.
	\end{equation} 
	We write for any $N \in \mathbb{R}_{\geq 0}$\index{Balls in Number Fields}
	\begin{equation}
		\label{eq: definition of BN}
		B_N := \{x \in \co_K: \|x\| \leq N\}.
	\end{equation} 
	\begin{example}
		If $ K = \q$ we clearly have $B_N = [-N,N]$. For another example, write $\zeta = e^{2\pi i /3}$ and say that $K = \q[\zeta]$. This is a cyclotomic field, so $\co_K = \z[\zeta]$. Taking some $x \in \co_K$, we have $x = a_1+a_2\zeta$ where $a_1,a_2\in\z$ and $$\sigma(x) = (a_1+a_2\zeta,a_1+a_2\zeta^2).$$ 
		From this, one easily computes that $$B_N = \{a_1+a_2\zeta\in \co_K:  [a_1^2+a_2^2-a_1a_2]^{1/2} \leq N\}.$$
	\end{example}

	Given some $x \in K$, we will also write $$B_N(x) := \{y \in \co_K: \|x-y\| \leq N\}.$$ 
	
	For any lattice $\Gamma \subset \co_K$, we write $$\lambda_i(\Gamma) := \inf\{N \in \mathbb{R}_{\geq 0}: B_N \text{ contains } i \text{ linearly independent elements of } \Gamma\}.$$ 
	Ideals of a number ring don't behave exactly like arbitrary lattices. Given functions $f$ and $g$, we will write $f(x) \ll_a g(x)$ if there is some constant $C_a$ depending only on the object $a$ (for example, the field $K$ or its degree $n$), such that $f(x) \leq C_ag(x)$ for all $x$. We will use the notation $f(x) \asymp_a g(x)$ if $f(x) \ll_a g(x)$ and $g(x) \ll_a f(x)$. One of the properties that distinguishes ideals from general lattices is that we always have that $\lambda_1(\gb) \asymp_K \lambda_n(\gb)$, as the next lemma demonstrates. The collection of lattices $\{k\z \times \z: k \in \mathbb{N}\}$ inside of $\z^2$ shows that this is not the case for general lattices, since we have $\lambda_1(k\z \times \z) = 1$ and $\lambda_2(k\z \times \z) = k$.
	
	\begin{lemma}
		\label{lm: lambda_i asymp N(gb)(1/n)}
		Let $K$ be a number field of degree $n$. Then for any ideal $\gb$ of $\co_K$, we have $$\lambda_1(\gb) \asymp_K \lambda_n(\gb) \asymp_K N(\gb)^{\frac{1}{n}}.$$ 
	\end{lemma}
	
	\begin{proof}
		This is Corollary 4 in \cite{Fraczyk}. They show that, writing $d_K$ for the discriminant of $K$, we always have $$d_K^{\max(0,\frac{1}{n}-\frac{1}{2m})}N(\gb)^{1/n} \ll_n \lambda_m(\gb) \ll_n d_K^{\min(\frac{1}{2n-2m+2},\frac{1}{n})}N(\gb)^{1/n}. $$ Since $d_K$ and $n$ only depend on $K$, the result follows.
	\end{proof}
	
	Notice that while \Cref{lm: lambda_i asymp N(gb)(1/n)} does not hold for lattices in general, Minkowski's second theorem (see Theorem 2E in \cite{Schmidt2}) shows that for any lattice $	\Gamma \subset \co_K$, \begin{equation}
		\label{eq: product of Lambdas}
		[\co_K:\Gamma] \asymp_K \lambda_1(\Gamma) \dots \lambda_n(\Gamma).
	\end{equation}

	\begin{definition}
		\label{def: Etale Algebra}
		An \index{\'Etale $\q-$Algebra}étale $\q-$algebra $K$ is a product of finitely many finite extensions of $\q$, that is, $K$ is an algebra that can be written as the cartesian product $K_1\times \dots \times K_m$ of number fields $K_i$ with coordinate-wise sum and product.
	\end{definition}
	
	Given an étale $\q-$algebra $K = K_1\times \dots \times K_m$, the ring of integers of $K$ is then $$\co_K = \co_{K_1} \times \dots \times \co_{K_m}.  $$ As in the number field case, every non-zero ideal of $\co_K$ can be uniquely factored into the product of prime ideals, with the prime ideals in $\co_K $ being of the form $P_1 \times \dots \times P_m$, where $P_i$ is a prime ideal of $\co_{K_i}$ for a unique $i$, and $P_j = \co_{K_j}$ for all remaining indexes. Hence, all ideals $I$ of $\co_K$ can be written in the form $I =I_1 \times \dots \times I_m$. We will say that $I$ is an \index{Invertible Ideal}\emph{invertible ideal}, if $I_i \neq (0)$ for every $1 \leq i \leq m$. This means that if $I$ is an invertible ideal, then $I$ is a lattice, and $[\co_K:I] < \infty.$
	
	Given ideals $I= I_1\times \dots \times I_m$ and $J= J_1\times \dots \times J_m$, we have $$I+J = (I_1+J_1) \times \dots \times (I_m+J_m) \text{ and } IJ = (I_1J_1) \times \dots \times (I_mJ_m). $$
	In particular, we say the ideals are coprime and write $(I,J) =1$, if $I+J = \co_K$, or equivalently, if $(I_i,J_i) = 1$ for every $i$. Similarly, we say that $I$ divides $J$  and write $I \mid J$ if $I_i \mid J_i$ for every $i$.  Notice that the Chinese Remainder Theorem also holds in this context, meaning that if $\gb_1,\dots, \gb_l$ are a sequence of pairwise coprime non-zero ideals of $\co_K$, then $$\co_K/\prod_{i=1}^l \gb_i \cong \prod_{i=1}^l \co_K/\gb_i,$$ with the isomorphism being given by the map that sends $x+\prod_{i=1}^l \gb_i $ into $(x+\gb_i)$. This is because if we write each $\gb_i$ as $\gb_i^{(1)}\times \dots \times\gb_i^{(j)}\times \dots \times\gb_i^{(m)}$, then we can break $\co_K/\prod_{i=1}^l \gb_i$ into the (Cartesian) product of $\co_{K_j}/\prod_{i=1}^l \gb_i^{(j)}$ over all $1 \leq j \leq m$. Applying the Chinese Remainder Theorem for each number field $K_j$, and then taking the Cartesian product of  $\prod_{i=1}^l\co_{K_j}/ \gb_i^{(j)}$ gives $ \prod_{i=1}^l \co_K/\gb_i$ 
	
	We will use the following lemma to count the points in a particular congruence class modulo a lattice. A proof is given in Proposition 3.2 of \cite{Bartnicka} for the case where $K$ is a number field and $\Gamma$ is an ideal of $\co_K$, but a proof of the more general case is exactly the same.
	
	\begin{lemma}
		\label{lm:Counting under restriction}
		Let $K$ be an étale $\q-$algebra of degree $n$ and $\Gamma$ a lattice of $\co_K$. For any $a \in \co_K$, define $$T_a(N) := |\{x\in B_N: x \equiv a \mod \Gamma \}|.$$ We have $$T_a(N) = \frac{|B_N|}{[\co_K:\Gamma]} + O_K\left(1+\sum_{i=1}^{n-1}\frac{N^i}{\lambda_1(\Gamma)\cdots\lambda_i(\Gamma)}\right).$$

	\end{lemma}

	\begin{remark}
		\label{rmk: particular case for ideals}
		In particular, if $\gb$ is a non-zero ideal in $\co_K$, then we have $$T_a(N) = \frac{|B_N|}{N(\gb)} + O_K\left(1+\sum_{i=1}^{n-1}\frac{N^i}{\lambda_1(\gb)\cdots\lambda_i(\gb)}\right).$$ Using \Cref{lm: lambda_i asymp N(gb)(1/n)}, we know that for every $i$, $\lambda_i(\gb) \asymp_K N(\gb)^{1/n}$ and so we have $$T_a(N) = \frac{|B_N|}{N(\gb)} + O_K\left(1+\left(\frac{N}{N(\gb)^{1/n}}\right)^{n-1}\right).$$
	\end{remark}

	\subsection{Dynamical Systems}
	
	We now present a number of standard results from the theory of dynamical systems that we will need. Given a group $G$ and a topological space $X$\index{Action by Homeomorphisms}, we say that a map $T:G \times X \rightarrow X$ is an action of $G$ on $X$ by homemorphisms, if the maps $T_g(x) := T(g,x)$ are always homeomorphisms, and they satisfy
	\begin{enumerate}
		\item $T_{\textbf{0}}(x) = x$ for all $x\in X$, where $\textbf{0}$ is the identity of $G$,
		\item$T_{g_1}(T_{g_2}(x)) = T_{g_1g_2}(x)$ for any $g_1,g_2 \in G$.
	\end{enumerate}

	\begin{definition}
		A \index{Dynamical System!Topological}topological dynamical system is a pair $(X,T)$, where $X$ is a compact metric topological space, and $T$ is an action of a group $G$ on $X$ by homeomorphisms.
	\end{definition}
	
	In general, we will only work with groups $G$ that are isomorphic to $\z^n$, so in this section $G$ is always assumed to be finitely generated,  abelian, and locally compact. In this section, we will denote the metric of $X$ by $d$. 
	
	A morphism $\phi$ of topological dynamical systems between $(X,T)$ and $(Y,S)$, is a continuous map from $X$ to $Y$ such that $\phi(T_g(x)) = S_g(\phi(x))$ for every $x \in X$ and $g \in G$. If $\phi$ is surjective, then we say that $\phi$ is a \index{Factor System}\textit{factor} morphism, and that $(Y,S)$ is a factor of $(X,T)$. Given a subset $A\subset X$, we say that $A$ is $T-$invariant if $T_g(A)=A$ for every $g \in G$.
	
	\begin{definition}
		\label{def: Joining}
		Given two topological  dynamical systems $(X,T)$ and $(Y,S)$, we say that a set $J \subset X \times Y $ is a \index{Joining!Topological}\textit{joining}, if it is closed, invariant under $(T\times S)$, and has full projections on both coordinates. If $ J \subsetneq X \times Y$, we say that $J$ is a non trivial joining (see Section 3 in \cite{Glasner}).
	\end{definition}

	We say that a space $(X,T)$ is \index{Minimal System}\textit{minimal}, if $X$ does not have a proper closed $T-$invariant subset. Equivalently, $X$ is minimal if the orbit $\{T_g(x):g \in G\}$ of every $x\in X$ is dense. Using Zorn's Lemma, one can easily show that every space has a non-empty minimal subsystem (see Theorem 4 in the first chapter of \cite{Auslander}).
	
	We say a system is \index{Dynamical System!Distal} \textit{distal} if for any $x\neq y$, we have $$\liminf_{g \in G} d(T_g(x),T_g(y)) > 0.$$
	On the other hand, we say that a system is \index{Dynamical System!Proximal}\textit{proximal} if for any $x,y \in X$, we have 
	\begin{equation}
		\label{eq: def of proximality}
		\liminf_{g \in G} d(T_g(x),T_g(y)) = 0.
	\end{equation} We will need the following two results about proximality. A proof of \Cref{lm: proximal implies unique fixed point} can be found in page 175 of \cite{Kerr}. 
	
	\begin{lemma}
		\label{lm: proximal implies unique fixed point}
		If the system $(X,T)$ is proximal, then there exists a unique $x_0$ such that $T_g(x_0) = x_0$ for all $g \in G$.
	\end{lemma}
	
	Given a set $S \subset G$, we say that it is \index{Syndetic Set}\textit{syndetic} if there is some compact set $K$ such that $S+K = G$. We have the following result (see \cite{Bartnicka}). 
	
	\begin{lemma}
		\label{lm: proximality auxil result}
		Let $(X,T)$ be a dynamical system for which there is some $x_0$ such that $T_g(x_0) = x_0$ for all $g\in G$. Then, the following are equivalent.
		
		\begin{enumerate}
			\item For all $x,y \in X$ and $\epsilon >0$, the set $\{g \in G: d(T_g(x),T_g(y)) < \epsilon\}$ is syndetic.
			\item For all $x \in X$ and $\epsilon >0$, the set $\{g \in G: d(T_g(x),x_0) < \epsilon\}$ is syndetic.
		\end{enumerate}

		If any of these conditions hold, then $(X,T)$ is proximal.
	\end{lemma}
	
	\begin{definition}
		\label{def: Ergodic topological dynamical system}
		We say $(X,T)$ is \index{Ergodic!Topological Dynamical System}\textit{ergodic} if every $T-$invariant closed proper subset of $X$ is nowhere dense (see Definition II.2 in \cite{Furstenberg}). We say it is \index{Dynamical System !Topological Weakly mixing} \textit{topologically weakly mixing} if $X \times X$ is ergodic.
	\end{definition}

	We now shift our focus to measure theoretical dynamical systems.
	
	\begin{definition}
		A \index{Dynamical System!Measure Theoretical}measure theoretical dynamical system is a quadruple $(X,\mathcal{F}, \mu, T)$, such that $(X,\mathcal{F}, \mu)$ is a probability space, and $T$ is an action of a group $G$ on $X$ by measurable functions that preserve $\mu$, that is, a map $T:G \times X \rightarrow X$ such that for every $A \in \mathcal{F}$ and $g\in G$, we have $$\mu(T_g^{-1}(A)) = \mu(A).$$
	\end{definition}
	
	Alternatively, we also say that $\mu$ is $T-$invariant, if $T$ preserves $\mu$. Most times we will write the dynamical system  $(X,\mathcal{F}, \mu, T)$ as $(X, \mu, T)$ instead, since $\mu$ will always be assumed to be defined on the Borel $\sigma-$algebra of $X$ (which is the $\sigma-$algebra generated by the open sets of $X$).
	
	A morphism between dynamical systems $(X,\mathcal{F}, \mu, T)$ and $(Y,\mathcal{G}, \nu, S)$ is a map $\phi:X' \rightarrow Y$ where $X'$ is a subset of $X$ such that $\mu(X') = 1$, satisfying the properties:
	
	\begin{enumerate}
		\item For all $x\in X'$ and $g \in G$, $\phi(T_g(x)) = S_g(\phi(x)).$
		\item The function $\phi$ is measurable, and $\nu = \phi_*(\mu)$,
	\end{enumerate}
	where the \index{Pushforward}\textit{pushforward} of a measure $\mu$ by a map $\phi$ is the measure defined by the equation $$\phi_*(\mu)(U) = \mu(\phi^{-1}(U))$$ for all measurable $U \in \mathcal{G}$. An isomorphism of the dynamical systems $(X,\mathcal{F}, \mu, T)$ and $(Y,\mathcal{G}, \nu, S)$ is a bijective map $\phi:X' \rightarrow Y'$, where $X'$ and $Y'$ are subsets of $X,Y$ satisfying $\mu(X') = 1$ and $\nu(Y') = 1$ respectively, for which (1) and (2) hold.
	
	If there is a morphism $\phi$ between the systems $(X,\mathcal{F}, \mu, T)$ and $(Y,\mathcal{G}, \nu, S)$, then we say that $Y$ is a factor of $X$. Notice that $(2)$ implies that $\nu(\phi(X)) = 1$, so whenever $\phi$ is a morphism, there must be sets $X'$ and $Y'$ of full measure such that $\phi$ restricted to $X'$ is surjective on $Y'$.
	
	\begin{definition}
		Given a topological dynamical system $(X,T)$, we say that a $T-$invariant Borel measure $\mu$ is \index{Ergodic!Measure}\textit{ergodic} if for any measurable set $U$ such that $$\mu(U \Delta T_g(U)) = 0$$ for every $g \in G$, we have $\mu(U) \in \{0,1\}$.
	\end{definition}
	
	Here, $\Delta$ denotes the symmetric difference $A\Delta B = (A \cup B)\setminus (A\cap B)$ for sets $A,B$. In particular, if for all $g \in G$ we have  $T_g(U) = U$ and $\mu$ is ergodic, then $\mu(U) \in \{0,1\}$. Notice that if $(X,\mathcal{F}, \mu, T)$  is an ergodic system (that is, a measure theoretical dynamical system such that $\mu$ is ergodic for $(X,T)$), and $(Y,\mathcal{G}, \nu, S)$ is one of its factors, then it must also be ergodic.
	
	\begin{definition}
		\label{def: measures in dynamical system}
		Given a topological dynamical system $(X,T)$, we write $\mathcal{M}(X)$\index{Space of Probability Measures} for the set of probability measures on $X$, and $\mathcal{M}^T(X)$ for the set of $T-$invariant probability measures on $X$. 
	\end{definition}
	
	The following is a classical result about $\mathcal{M}^T$ (see Theorem 8.4 in \cite{Ward}). 
	
	\begin{theorem}
		\label{thm: ergodic is extremal}
		The set  $\mathcal{M}^T(X)$ is a non-empty convex closed subset of $\mathcal{M}(X)$. Additionally, a measure $\mu$ is an extremal point of $\mathcal{M}^T(X)$ if and only if it is ergodic. 
	\end{theorem}
	
	By an extremal point, we mean a point $z$ such that if $z$ can be written as $z = \alpha x + \beta y$ with $\alpha+\beta = 1$, $\alpha,\beta >0$, then $z = x =y$.
	
	The fact that  $\mathcal{M}^T(X)$ is always non-empty is known as the Krylov–Bogolyubov theorem\index{Krylov–Bogolyubov Theorem} (see Corollary 4.2 in \cite{Ward}). As a corollary of \Cref{thm: ergodic is extremal}, note that for every dynamical system $(X,T)$, there must exist some ergodic measure.

	A particularly important class of dynamical systems are minimal rotations. Let $X$ be a compact metrizable abelian group with identity element $\textbf{0}$ and Haar measure $m_X$. We call a topological dynamical system $(X,T)$ where $T_{g}(x) = T_g(\textbf{0})+x$ a \index{Minimal Rotation of a Compact Group}\textit{rotation} of $X$. We write $m_X$ for the Haar measure on $X$ and $\textbf{0}$ for the identity in $X$. Since $m_X$ is invariant under all translations, it is $T-$invariant. We say a system $(X,T)$ is \index{Uniquely Ergodic System}\textit{uniquely ergodic}, if $M(X)$ contains a unique measure (by  \Cref{thm: ergodic is extremal}, this measure must be ergodic). The following result shows that a rotation $(X,T)$ is uniquely ergodic if and only if it is a minimal system. In Theorem 4.14 of \cite{Ward} a proof is shown in the case $G= \z$, which easily generalizes for an abelian locally compact $G$.
	
	\begin{theorem}
		\label{thm: minimal rotation of group}
		Let $(X,T)$ be a rotation of $X$. The following are equivalent.
		\begin{enumerate}
			\item The system $(X,T)$ is uniquely ergodic.
			\item $T$ is ergodic for $m_X$.
			\item The set $\{T_g(\textbf{0}):g \in G\}$ is dense in X. 
		\end{enumerate}
	\end{theorem}
	
	 Throughout we will be working with Følner sequences in order to define densities.
	
	\begin{definition}
		\label{def: Følner sequence}
		We say that a sequence of finite, non-empty sets $I_N \subset G$, $N \geq 1$, is a \index{Følner Sequence}\textit{Følner sequence} if for every $x\in G$,  $$\lim_{N \rightarrow \infty} \frac{|(x+I_N) \Delta I_N|}{|I_N|} = 0.$$
	\end{definition}

	As an example, when $G=\z$, the sequence $I_N = [0,N]$ is a Følner sequence. Additionally, notice that if $I_N$ is any Følner sequence, and $F_N$ is given by $F_N = a_N+ I_N $ for some $a_N \in \co_K$, then given $x\in \co_K$ we have $$\lim_{N \rightarrow \infty} \frac{|(x+F_N) \Delta F_N|}{|F_N|} = \lim_{N \rightarrow \infty} \frac{|a+((x+I_N) \Delta I_N)|}{|I_N|} = 0,$$ and so $F_N$ is again a Følner sequence.
 Given a topological dynamical system $(X,T)$, we will write $\htop(X,T)$ for its topological entropy, and given a measure theoretical dynamical system $(X,T,\mu)$, we will denote its measure theoretical entropy by $h(X,T,\mu)$. For definitions, see Chapter 9 of \cite{Kerr}.
	
	When $X$ is a shift space\index{Shift Space}, that is, $X \subset \mathcal{A}^{G}$ for some finite set $\mathcal{A}$ and $T$ is defined by $$(T_ax)(m) = x(m-a),$$  then we can consider its \index{Entropy!Patch Counting}\textit{patch counting entropy}. Let  $x|_{I}$ denote the restriction of the function $x:G \rightarrow \mathcal{A}$ to some $I \subset G$. The patch counting entropy is given by $$\hpc(X,T) = \limsup_{N\rightarrow\infty}\frac{\log(|\{x|_{-I_N}: x \in X \}|)}{|I_N|},$$  independent of the the chosen Følner sequence $I_N$.  We have the following result (see Example 9.41 in \cite{Kerr}).
	
	\begin{lemma}
		\label{lm:Topological and patch  counting are same}
		If $(X,T)$ is a shift space, then $\hpc(X,T) = \htop(X,T)$.
	\end{lemma}
	
	Topological and measure theoretical entropies are related by the Variational Principle (see Theorem 9.48 in \cite{Kerr}).
	
	\begin{theorem}
		\label{thm: variational principle}
		Let $(X,T)$ be a topological dynamical system. Let $\mathcal{M}^T(X)$ be the set of $T-$invariant measures (as in \Cref{def: measures in dynamical system}). We have $$\htop(X,T) = \limsup_{\mu \in \mathcal{M}^T(X)} h(X,T,\mu).$$
	\end{theorem}
	
	We finish this section with the definition of a \index{Generic Point}\textit{generic} point.
	
	\begin{definition}
		\label{def: generic point}
		Let $(X,T,\mu)$ be a dynamical system and $I_N$ a Følner sequence.  We say that $x\in X$ is generic with respect to $I_N$ if for every continuous function $f \in C(X)$ we have $$\lim_{N \rightarrow \infty} \frac{1}{|I_N|}\sum_{a\in I_N} f(T_a(x)) = \int_X f \; d \mu.$$
	\end{definition}
	
	The Ergodic Theorem gives conditions for almost every point in a system to be generic. The following very general form was shown by Lindenstrauss in \cite{Lindestrauss}. We say that a Følner sequence $I_N$ is \index{Følner Sequence!Tempered}\textit{tempered} if there is some constant $C$ such that for all $N$,
	\begin{equation}
		\label{eq:tempered Følner sequence}
		\left|\bigcup_{L<N} -I_L + I_N\right| < C|I_N|.
	\end{equation}  Note that every Følner sequence has a tempered subsequence (see Proposition 1.4 in \cite{Lindestrauss}). The Pointwise Ergodic Theorem states the following.
	
	\begin{theorem}
		\label{thm: ergodic theorem}
		Let $(X,T,\mu)$ be an ergodic dynamical system, and $I_N$ a tempered Følner sequence. Then, for every $f\in L^1(X)$, we have $$\lim_{N \rightarrow \infty} \frac{1}{|I_N|} \sum_{a\in I_N} f(T_a(x)) = \int_X f d \mu$$ for $\mu-$almost ever $x \in X$.
		
	\end{theorem}  
	
	\begin{remark}
		\label{rmk: ergodic measures are mutualy singular}
		In particular, we have that given an ergodic measure $\mu$ and a tempered Følner sequence, the set $\Gen(\mu,I_N)$ of generic points with respect to $I_N$ satisfies $\mu(\Gen(\mu,I_N)) = 1$ (see Corollary 8 in \cite{Host}). 
	\end{remark}
	
	The following lemma is classic and follows easily from Prokhorov's theorem (see Theorems 5.1 and 5.2 in \cite{Billingsley}).
	
	\begin{lemma}
		\label{lm: In uniquely ergodic every point generic}
		Let $(X,T)$ be a uniquely ergodic system, with unique measure $\mu$. Then, for every Følner sequence $I_N$, $$\Gen(\mu,I_N) = X.$$
	\end{lemma}

	\section{Sieves and $R-$free numbers}
	
	In this section we introduce the basic objects with which we are going to work. Throughout, we assume that $K$ is an étale $\q-$algebra of degree $n$. 
	
	\subsection{Shift Spaces}

	We start with the definition of sieve.

	\begin{definition}
		\label{def: Sieve}
		A \index{Sieve}sieve over an étale $\q-$algebra $K$ is a pair $(\cb_R,(R_\gb)_{\gb \in \cb_R})$ where $\cb_R$ is an infinite set of pairwise coprime invertible ideals of $\co_K$, and each $R_\gb$ is a set of the form $R_\gb =  S_\gb + \gb $ with $S_\gb$ finite sets such that $R_\gb \neq \co_K$ for every $\gb \in \cb_R$.
	\end{definition}
	
	We say that a sieve $R$ is \textit{supported}\index{Support of a Sieve} on the set $\cb_R$. Note that we allow $R_\gb$ to be empty for any number of ideals $\gb$, as long as there are still infinitely many ideals $\gb$ such that $R_\gb \neq \emptyset$. Yet, when proving results we will always assume that $R_\gb \neq \emptyset$ for all $\gb \in \cb_R$.  In many cases it will be convenient to assume that $\cb_R$ is ordered, so that $\cb_R= \{\gb_1,\gb_2,\dots\}$, in which case we will write $R_i = S_i +\gb_i$. As an example, let $\cb = \{p_i^2:i\in \mathbb{N}\}$, where $p_i$ denotes the $i$-th prime. We can describe the squarefree sieve as the sieve $R$ supported on $\cb$ and defined by $R_{i} = p_i^2\z$. 
	
	A $\cb-$free system\footnote{We will always use the term $\cb-$free system to mean sieves that can be obtained by sieving out only ideals. This is different from the usual literature, where the term refers to the associated topological or measure theoretical dynamical systems instead. }\index{$\cb-$free System} denotes a sieve $R$ such that $R_\gb = \gb$ for every $\gb \in \cb_R$. We will not consider generalizations of sieves where one allows for $\cb_R$ to be a set of non-coprime ideals. Since this case has been treated in the literature for $\cb-$free
	systems (see for example \cite{Dymek}), we still sometimes feel the need to refer to this case. Therefore, we will use the term \index{Pseudosieve}\textit{pseudosieve} for a pair $(\cb_R,(R_\gb)_{\gb \in \cb_R})$ where $\cb$ is an infinite set of (not necessarily pairwise coprime) invertible ideals and every $R_\gb$ can be written as $A_\gb + \gb$ for some $A_\gb$, with $R_\gb \neq \co_K$.

	Given a sieve $R$ and some ideal $\gb$, we will write $|R_\gb|$ for the cardinality of the image of $R_\gb$ in $ \co_K/\gb$. This means that its complement $R_\gb^c = \co_K \setminus R_\gb$ satisfies $|R^c_\gb| = N(\gb)-|R_\gb|$. More generally, when $A$ is a subset of $\co_K$, we will write $|A+\gb|$ for the cardinality of the image $A$ in $\co_K/\gb$. We then define the \textit{volume} of a set $A \subset \co_{K}/\gb$ as $\vol(A):= |A+\gb|/N(\gb)$.
	
	\begin{definition}
		We say a sieve $R$ is \index{Sieve!Erdős}Erdős if $$\sum_{\gb \in \cb_R} \vol(R_\gb) < \infty.$$
	\end{definition}
	
	These are so named because in the $\cb-$free system literature, a system is called \textit{Erdős} if the elements of $\cb$ are pairwise coprime and $\sum_{\gb\in \cb}1/N(\gb) < \infty$, after Erdős studied such systems (over $\q$) in \cite{Erds}. 
	
	One of the main purposes of our work is to study sets which can be realized as the $R-$free numbers for some sieve, that is, those that are not sieved by any $R_\gb$. 
	
	\begin{definition}
		Given a sieve $R$ over an étale $\q-$algebra $K$, we denote by the $R-$free numbers\index{$R$-free Numbers} the set $$\fr := \co_K\setminus\bigcup_{\gb \in \cb_R} R_\gb.$$

	\end{definition}
	
	For example, let $R$ be a sieve such that $\cb_R = \{p^k\z:p \text{ prime}\}$ and defined by $R_p = p^k\z$ for some integer $k \geq 2$. Then $\fr$ corresponds to numbers that are not divisible by $p^k$ for any prime, that is, the $k-$free numbers.
	
	Throughout we will identify the space $\{0,1\}^{\co_K}$ with the set of subsets of $\co_K$ (by identifying the indicator function $\one_A$ with the set $A$). This set is a compact topological space equipped with the product topology on countably many copies of the discrete space $\{0,1\}$. Note that this topology is metrizable, and is induced by the metric $d$ such that for $X,Y \subset \co_K$, we have $$d(X,Y) = \min(1,\sup\{\epsilon>0: X\cap B_\epsilon = Y \cap B_\epsilon\}^{-1} ).$$ 
	That is, under this metric, two sets are close if they are identical in a big neighborhood around the origin. We define the shift map $S$ in $\{0,1\}^{\co_K}$ as the action of $\co_K$ on this space given by \begin{equation}
		\label{eq: definition of S}
		S_a(A) = A-a.
	\end{equation}
	
	We will define dynamics on two \textit{shift spaces}\index{Shift Space}, which are closed, $S-$invariant subsets of $\{0,1\}^{\co_K}$. The first space we will consider is the set $\Omega_R$ of $R-$\textit{admissible sets}.
	
	\begin{definition}
		Given a sieve $R$ over an étale $\q-$algebra $K$, we say a set $A \subset \co_K$ is $R-$\textit{admissible}\index{Admissible Set} if for every $\gb \in \cb_R$, $-A+R_\gb \neq \co_K$.
	\end{definition}
	
	Equivalently, we say that $A$ is an $R-$admissible set if for every $\gb \in \cb_R$ there is some $\delta_\gb \in \co_K$ such that  $(A+\delta_\gb)\cap R_\gb = \emptyset$. By taking $\delta_\gb =0$ for every $\gb$, we see that $\fr$ is always $R-$admissible. Additionally, it is clear that if $A$ is $R-$admissible, so is $S_a(A)$ for any $a\in \co_K$. 
	
	\begin{definition}
		\label{def: XR and OmegaR}
		The set $\Omega_R \subset \{0,1\}^{\co_K}$ denotes the set of all $R-$admissible sets\index{Space of Admissible Sets}. We write $X_R$ \index{Orbit Closure ! of $R-$free integers} for the orbit closure of $\fr$ in $\{0,1\}^{\co_K}$, that is, $$X_R = \overline{\{S_a(\fr):a \in \co_K\}}.$$
		
	\end{definition}
	
	Since $X_R$ is a closed subset of a compact space, it is also compact. The same holds for $\Omega_{R}$, as we show in the following lemma.
	
	\begin{lemma}
		$\Omega_R$ is closed.
	\end{lemma}
	
	\begin{proof}
		Assume that $A$ is not an admissible set. By definition, there must be some $\gb \in \cb_R$ such that $-A+R_\gb = \co_K$. Let $N$ be big enough that $(-A\cap B_N)+\gb =  \co_K$. For any admissible set $B$, we necessarily have $(-B\cap B_N) \neq (-A\cap B_N)$, so $d(A,\Omega_R) > 1/N$.
	\end{proof}

	\begin{remark}
		This implies that  $X_R \subset \Omega_R$, given that $\{S_a(\fr):a \in \co_K\} \subset \Omega_{R}$, and so the closure of the orbit must also be contained in $\Omega_{R} $, as this space is closed.
	\end{remark}
	
	Since $X_R$ and $\Omega_{R}$ are both compact, $(\Omega_R,S)$ and $(X_R,S)$ are topological dynamical systems. 
	
	\begin{definition}
		\label{def: GR}
		Given a sieve $R$ over an étale $\q-$algebra $K$, we define its \textit{odometer}\index{Odometer} as the dynamical system ($G_R,T,\mmp$), where $$G_R := \prod_{\gb \in \cb_R} \co_K /\gb,$$ the action $T$ of $\co_K$ on $G_R$ is given by $$T_a((g_\gb)_{\gb \in \cb_R}) = (a+g_\gb)_{\gb \in \cb_R},$$ and $\mmp$ is the Haar measure on $G_R$.
	\end{definition}
	
	The measure $\mmp$ has a simple description. Given a set $A =\{(\gb_1,a_1),\dots,(\gb_m,a_m)\}$ where the $\gb_i \in \cb_R$ are distinct and $a_i \in \co_K$, let $C_A$ be the cylinder set\index{Cylinder Sets} $$C_A = \{g \in \gr: \mathlarger \forall_{1 \leq i \leq m} \hspace{5pt} g_{\gb_i} \equiv a_i \mod \gb_i \}.$$ These sets generate the Borel $\sigma-$algebra of $\gr$, and for any such $A$, $\mmp(C_A)$ is given by $$\mmp(C_A) = \prod_{i=1}^m \frac{1}{N(\gb_i)}.$$
	
	Fixing some ordering of $\cb_R = \{\gb_1, \gb_2, \dots\}$, and for any $g \in \gr$,  writing $g_i := g_{\gb_i}$, we can define a metric on $\gr$ by $$d(g,h):= \sum_i \frac{\rho_i(g,h)}{2^i},$$ where $\rho_i(g,h) = 0 $ if $g_i = h_i$ and $1$ otherwise. We will write $\textbf{0}$ for the identity of $\gr$. For any integer $L \geq 1$ and tuple $(g_1,\dots, g_L)$, the Chinese Remainder Theorem states that there is some $a \in \co_K$, such that $T_a(\textbf{0})_i = g_i$ for any integer $1 \leq i \leq L$.  Therefore, it is clear that $\overline{\{T_a(\textbf{0}):a \in \co_K\}} = \gr$, which means that   ($G_R$,$T$) is a rotation of a group, and so a uniquely ergodic system by  \Cref{thm: minimal rotation of group}.
	
	We now define a measure on $\Omega_R$, known as the Mirsky measure\index{Mirsky Measure} (since this measure will associate to a set $A$ the frequency of the pattern $A$ in $\fr$, which Mirsky first considered in \cite{Mirsky} for the $r-$free numbers). In order to do this, we consider the map $\varphi_R:G_R \rightarrow \Omega_R$, which we will define by the relation \begin{equation}
		a \in \varphi_R(g) \Leftrightarrow \mathlarger \forall_{\gb \in \cbr} \hspace{5pt} a+g_\gb \not \in R_\gb.
	\end{equation}
	Note that the image of $\varphi_R$ does indeed lie in $\Omega_R$,  since by definition $(\varphi_R(g) + g_\gb) \cap R_\gb = \emptyset$ for every $\gb \in \cbr$. The elements of $\varphi_R(\textbf{0})$ are those $a$ such that $a \not \in R_\gb$ for every $\gb \in \cbr$, that is, $$\varphi_R(\textbf{0}) = \fr.$$
	
	\begin{definition}
		\label{def: Mirsky measure}
		Given a sieve $R$, we define the \textit{Mirsky measure} $\nu_R$ on $\Omega_R$ as the pushforward of $\mmp$ in $\Omega_R$ by $\varphi_R$, that is $\nu_R(A) := \mmp(\varphi_R^{-1}(A))$ for any measurable set $A$.
	\end{definition}
	
	This measure is defined on the Borel $\sigma-$algebra of $\Omega_{R}$, which is generated by \textit{cylinder sets}. 
	\begin{definition}
		Given a sieve $R$ and disjoint sets $A,B$, we define the cylinder set $C^R_{A,B}$ as the set $$C^R_{A,B} := \{Y\in \Omega_R: A \subset Y,  Y \cap B = \emptyset\}.$$
	\end{definition}
	
	The topology in $\Omega_R$ is generated by such sets with $A$ and $B$ finite, so they are open. Simultaneously, if $A$ and $B$ are finite, $ C^R_{A,B}$ is closed. Indeed, if $Y_m$ is a sequence of sets in $ C^R_{A,B}$ converging to some $Y\in \Omega_{R}$, and $N$ is so big that $ A \cup B \subset B_N$, then $d(Y,Y_m) < 1/N$ implies that $Y$ and $Y_m$ agree in $B_N$, so $A \subset Y$ and $B \cap Y = \emptyset$, so $Y \in C^R_{A,B}$. That is, these $ C^R_{A,B}$ are clopen (sets that are both open and closed), so it follows that $\one_{C_{A,B}}$ is a continuous function whenever $A$ and $B$ are finite.

	In particular, $C^R_{A,\emptyset}$ corresponds to the $R-$admissible sets that contain $A$. In this case, $\nu_R$ is given by	
	\begin{equation}
		\label{eq:Formula for nu_R with B empty}
		\nu_R(C^R_{A,\emptyset}) = \mmp(\varphi_R^{-1}(C_{A,\emptyset})) = \mmp(\{g \in G_R: \mathlarger \forall_{\gb \in \cbr} \hspace{5pt} (g_\gb + A) \cap R_\gb = \emptyset\})=\prod_{\gb \in \cbr} \left(1-\frac{|-A+R_\gb|}{N(\gb)}\right).
	\end{equation}
	
	Using the following lemma (which is Lemma 2.3 from \cite{Abdalaoui}), it follows that when $A$ and $B$ are finite we have
	\begin{equation}
		\label{eq:Formula for nu_R} 	
		\nu_R(C^R_{A,B}) = \sum_{A \subset D \subset  A\cup B} (-1)^{|D\setminus A|} \prod_{\gb \in \cbr} \left(1-\frac{|-D+R_\gb|}{N(\gb)}\right). 
	\end{equation}
	
	\begin{lemma}
		\label{lm:aux for FR is generic}
		Let $\mu$ be a probability measure is a space $\Omega$, and let $E_x$ be a sequence of measurable sets. Denote $F_x := \Omega\setminus E_x$. Then, given finite disjoint sets $A$ and $B$,  we have $$\mu\left(\bigcap_{a\in A}E_a \cap\bigcap_{b\in B}F_b\right) = \sum_{A \subset D \subset  A\cup B} (-1)^{|D\setminus A|} \mu\left( \bigcap_{d \in D} E_d  \right).  $$
	\end{lemma}
	
	We will now study some of the properties of the system $(\Omega_R,S,\nu_R)$. We start with the following result.
	
	\begin{lemma}
		\label{lm: varphi is factor}
		The dynamical system $(\Omega_R,S,\nu_R)$ is a factor of ($G_R$,$T$,$\mmp$). 
	\end{lemma}
	
	\begin{proof}
		We show that $\varphi_R$ is the factor map between these systems. It is a measure preserving map by the definition of $\nu_R$. So we are left with showing that $\varphi_R \circ T = S \circ \varphi_R$. This corresponds to showing that for any $x,a \in \co_K$, $a \in \varphi_R(g+x)$ if and only if $a \in -x + \varphi_R(g)$. But $a \in \varphi_R(g+x)$ is equivalent to $(a+x) + g_\gb \not \in R_\gb$ for all $\gb \in \cbr$, which is equivalent to $(a+x) \in  \varphi_R(g)$, so the result holds.
	\end{proof}

	Since ($G_R$,$T$,$\mmp$) is ergodic, it follows that $(\Omega_R,S,\nu_R)$ also is. Additionally, since ($G_R$,$T$,$\mmp$) is a minimal rotation of a compact group, we have $h(G_R,T,\mmp) = 0$ (see \Cref{def: measures in dynamical system} for the definition of measure theoretical entropy). It follows that \begin{equation}
		\label{eq:entropy is 0}
		h(\Omega_R,S,\nu_R) = 0.
	\end{equation}
	
	\begin{definition}
		\label{def: density}
		Given a set $A \subset \co_K$ and a Følner sequence $I_N$, we define the upper and lower \index{Density}densities with respect to $I_N$ to be respectively,  $$\overline{d}_I(A) := \limsup_{N \rightarrow \infty} \frac{|A \cap I_N|}{|I_N|} \text{ and } \underline{d}_I(A) := \liminf_{N \rightarrow \infty} \frac{|A \cap I_N|}{|I_N|}.$$ If these agree, we write the limit as $d_I(A)$, which we call the density of $A$ with respect to $I_N$.
		
		In the case where $I_N = B_N$ (as defined in \Cref{eq: definition of BN}), we omit the $I$, and write $\overline{d}(A),\underline{d}(A),d(A)$ for the corresponding densities.
	\end{definition}
	
	Since we will be interested in computing the densities of $\fr$ for different sieves $R$, we will use the following lemma which allows us to compute the density of the elements that miss a set of congruence classes.

	\begin{lemma}
		\label{lm: equidistribution for Følner sequences}
		Let $L \geq 1$ be an integer,  $I_N$ a Følner sequence, $\gb_1,\dots,\gb_L$ a collection of pairwise coprime ideals. For each $i$, take $A_i \subset \co_K$ and let $R_i = A_i + \gb_i$. If  $$C_L := \{x \in \co_K: \mathlarger \forall_i \hspace{2pt} x\not \in R_i \},$$ then $$\frac{1}{|I_N|}\sum_{a\in I_N} \one_{C_L}(a) \rightarrow d_I(C_L) = \prod_{i=1}^L\left(1-\frac{|R_i|}{N(\gb_i)}\right),$$ as $N \rightarrow \infty$.
	\end{lemma}
	
	\begin{proof}
		
		Let $\gc_L$ denote the product of all the $\gb_i$, and consider the topological dynamical system $(X,T)$ where $X = \co_K/\gc_L$ and $T_a(x) = x+a$ is an action of $\co_K$ on $X$. Since $(X,T)$ is a minimal rotation of a compact group,  \Cref{thm: minimal rotation of group} implies that it is uniquely ergodic. The unique $T-$invariant probability measure is given by $$\mmp(A) = \frac{|A|}{N(\gc_L)}.$$
		
		Let $\frak{C}_L$ for the image of $C_L$ in $X$. By the Chinese Remainder Theorem, there is a bijection between $\frak{C}_L$ and tuples $(x_1,\dots,x_L)$ with $x_i \in \co_K/\gb_i$ such that $x_i \not \in R_i$ for every $i$. Therefore $|\frak{C}_L| = \prod_i(N(\gb_i)- |R_i|)$, and so $$\mmp(\frak{C}_L) = \frac{|\frak{C}_L|}{N(\gc_L)} = \prod_{i=1}^L\left(1-\frac{|R_i|}{N(\gb_i)}\right).$$ By  \Cref{lm: In uniquely ergodic every point generic} the point $\textbf{0}$ belongs to $\Gen(\mmp,I_N)$ for any Følner sequence $I_N$, that is, $$\frac{1}{|I_N|}\sum_{a\in I_N} \one_{C_L}(a) = \frac{1}{|I_N|}\sum_{a\in I_N} \one_{\frak{C}_L}(T_a(\textbf{0}))  \rightarrow \mmp(\frak{C}_L) =\prod_{i=1}^L\left(1-\frac{|R_i|}{N(\gb_i)}\right)$$ as $N$ goes to infinity.	
	\end{proof}
	
	We want to know when the convergence in  \Cref{lm: equidistribution for Følner sequences} holds as we take $L$ to infinity. We now introduce the notion of weak light tails, which we will show is exactly the condition which determines whether this holds or not. This will be the key ingredient to generalizing point $(1)$ of Sarnak's program to Erdős sieves.
	
	\begin{definition}
		We say a sieve $R$ supported on $\cb_R=\{\gb_1,\gb_2,\dots\}$ has \index{Weak Light Tails}\emph{weak light tails} for the Følner sequence $I_N$ (or, with respect to $I_N$)  if it is Erdős, and
		$$\lim_{L \rightarrow \infty} \overline{d}_I\left(\bigcup_{i >L } R_i \setminus \bigcup_{j \leq L } R_j\right) = 0.$$
	\end{definition}

	Equivalently, we say an Erdős sieves has weak light tails for $I_N$ if  $$ \lim_{L \rightarrow \infty}\limsup_{N \rightarrow \infty}\frac{|\{x \in I_N: x \in R_i \text{ for some } i \text{ with } i > L, x  \not \in R_i \text{ if } i \leq L\}|}{|I_N|} = 0. $$  In \Cref{prop: sieve without density}, we provide an example of a sieve $R$ such that $\fr$ does not have a density with respect to $B_N$. This shows that the limit supremum used on this definition is required as sets of form $\bigcup_{i >L } R_i$ may not have a density.
	
	We say that such sieves have \emph{weak} light tails, because we will also consider another property of sieves, which we denote by \emph{strong light tails}. As the name implies, a sieve with strong light tails will have weak light tails.
	
	\begin{definition}
		We say a sieve $R$ has \index{Strong Light Tails}\emph{strong light tails} for a Følner sequence $I_N$ (or, with respect to $I_N$), if it is Erdős and $$\lim_{L \rightarrow \infty} \overline{d}_I\left(\bigcup_{i >L } R_i\right) = 0.$$
	\end{definition}
	
	Equivalently, we say an Erdős sieve $R$ has strong light tails for $I_N$ if $$ \lim_{L \rightarrow \infty}\limsup_{N \rightarrow \infty}\frac{|\{x \in I_N: x \in R_i \text{ for some } i \text{ with } i > L\}|}{|I_N|} = 0. $$ 
	Clearly strong light tails imply weak light tails, but as we will show in \Cref{ex:Sieve with weak but not strong light tails}, the converse does not hold. Although the definition of these properties assumes that $\cb_R$ was ordered, both properties are invariant under reordering of $\cb_R$, as we will show in \Cref{lm:Light tails invariant under } and \Cref{lm:Strong light tails is invariant under permutation}.
	
	Our use of the term \textit{light tails} to describe such sieves is inspired by the use of the term in Section 1.1 of \cite{Dymek}. In this paper the authors defined a $\cb-$free system with light tails as a set $\{b_1,b_2,\dots\}\subset \mathbb{N}$ such that $$\lim_{L \rightarrow \infty}\overline{d}\left(\bigcup_{i>L}b_i\z\right) = 0.$$ 
	It is clear that all such sets correspond to sieves with strong light tails.
	
	\begin{example}
		\label{ex:Sieves with and without light tails}
		
		For an example of a sieve without weak light tails for any $I_N$, let $R$ be the sieve over $\q$ and supported on $\cb = \{p_i^2: i \in \mathbb{N}\}$, defined by $$R_{i} =  \{-i,i\}+ p^2_i\z.$$ We have for any $x \in \z\setminus\{0\}$ that  $x\in R_{|x|}$, so given $L \geq 1$ we have $$|I_N \cap \bigcup_{i > L} R_i| \geq |I_N|-L.$$ On the other hand, by  \Cref{lm: equidistribution for Følner sequences}, we have that $$|I_N\setminus \bigcup_{i \leq L} R_i | = |I_N|\prod_{i=1}^L \left(1-\frac{2}{p_i^2}\right) +o(|I_N|).$$ Consequently, we have that $$\lim_{N \rightarrow \infty} \frac{|I_N \cap \bigcup_{i > L} R_i\setminus\bigcup_{j \leq L} R_j|}{|I_N|} \geq \lim_{N \rightarrow \infty} \frac{\left(|I_N|\left(1-\prod_{i=1}^L \left(1-\frac{2}{p_i^2}\right)\right) -L -o(|I_N|)\right)}{|I_N|} = 1-\prod_{i=1}^L \left(1-\frac{2}{p_i^2}\right)$$ which is bigger than $0$ for every $L$. Therefore the sieve $R$ does not have weak light tails for any Følner sequence.
		
		On the other hand, let $I_N = [0,N]$, and consider the sieve $R$ given by $$R_i = p_i^2\z,$$ that is, the squarefree sieve. It has strong light tails for $I_N$ as a consequence of  \Cref{lm: light tails for B_n finite}. Note that for every sieve $R$ there is some Følner sequence $I_N$ such that $R$ does not have weak light tails with respect to $I_N$, as we will show in \Cref{rmk: no Erdős sieve has weak light tails for every Følner sequence}.
	\end{example}
	
	We are now in a position to generalize point $(1)$ of Sarnak's program to Erdős sieves, by expliciting when $\fr$ is generic in $(\Omega_{R},S, \nu_R)$. We start by showing two lemmas. The following very simple lemma shows that point $(1)$ of Sarnak's program is a theorem of interest from a number theoretic point of view.
	
	\begin{lemma}
		\label{lm:Sums to cardinality of sets}
		Let $A,B,X$ be subsets of $\co_K$. We have
		$$ \sum_{x\in X}\one_{C^R_{A,B}}(S_x (\fr)) = |\{x\in X: x+A \subset \fr, (x+B) \cap \fr = \emptyset\}|.$$
	\end{lemma}
	
	\begin{proof}
		If $\one_{C^R_{A,B}}(S_x (\fr))$ is $1$, then  $-x+\fr$ is an admissible set containing $A$ and disjoint from $B$, which is equivalent to $x+A \subset \fr$ and $x+B \cap \fr = \emptyset$, since $S_x(\fr)$ is always admissible.
	\end{proof}
	
	This allows us to relate the weak light tails of $R$ and $d_I(\fr)$, which we  will use to show our generalization of point $(1)$ of Sarnak's program.
	
	\begin{lemma}
		\label{LM: dfr smaller than nu and equivalence}
		Let $R$ be a sieve and $I_N$ a Følner sequence. Then $$\bar{d}_I(\fr) \leq \nu_R(C^R_{\{0\},\emptyset}) =  \prod_{\gb \in \cbr}  \left(1-\frac{|R_\gb|}{N(\gb)}\right).$$ For an Erdős sieve $R$, $d_I(\fr)$ exists and equals $\nu_R(C^R_{\{0\},\emptyset})$ if and only if $R$ has weak light tails for $I_N$.
	\end{lemma}
	
	\begin{proof}
		
		Notice that $a \in \fr$ is equivalent to $a+\{0\} \subset \fr$. Therefore, by  \Cref{lm:Sums to cardinality of sets}, we have  $$ \frac{|I_N \cap \fr|}{|I_N|} =  \frac{|\{a \in I_N: a + \{0\} \subset \fr\}|}{|I_N|} =\frac{1}{|I_N|}\sum_{a \in I_N} \one_{C^R_{\{0\},\emptyset}}(S_a(\fr)).$$
		Since $\varphi_R(\textbf{0}) = \fr$ and $\varphi_R$ is a factor map (by  \Cref{lm: varphi is factor}), we have that $$S_a(\fr) = S_a(\varphi_R(\textbf{0})) = \varphi_R(T_a(\textbf{0})),$$ and so $S_a(\fr) \in  C^R_{\{0\},\emptyset}$ is equivalent to $T_a(\textbf{0}) \in \varphi_R^{-1}(C^R_{\{0\},\emptyset})$, hence we get
		
		\begin{equation}
			\label{eq: density and sums with T}
			\frac{|I_N \cap \fr|}{|I_N|}  = \frac{1}{|I_N|}\sum_{a \in I_N} \one_{\varphi_{R}^{-1}(C^R_{\{0\},\emptyset})}(T_a(\textbf{0})).
		\end{equation}

		We would like to evaluate the sum in the right hand side as $N$ goes to infinity using  \Cref{lm: In uniquely ergodic every point generic}, given that $(G_R,T)$ is a uniquely ergodic system. The problem is that $\varphi_{R}$ is not continuous, so $\one_{\varphi_{R}^{-1}(C^R_{\{0\},\emptyset})}$ might not be continuous. We must therefore approach this set through clopen sets.  Let us order the support $\cbr = \{\gb_1,\gb_2,\dots\}$. We have that $$\varphi_{R}^{-1}(C^R_{\{0\},\emptyset}) = \{g \in G_R: g_i \not \in R_i \text{ for all } i  \}.$$
		
		By defining  $$C_L := \{g \in G_R: g_i \not \in R_i \text{ for all } i \leq L\},$$ we have that $C_L$ is a clopen subset of $G_R$, so $\one_{C_L}$ is a continuous function on $G_R$. By the same argument used in  \Cref{lm: equidistribution for Følner sequences} (simply choosing $X$ to be $G_R$ instead), we have $$ \lim_{N\rightarrow\infty}\frac{1}{|I_N|}\sum_{a \in I_N}\one_{C_L}(T_a(\textbf{0})) = \mmp(C_L) = \prod_{i =1 }^L  \left(1-\frac{|R_i|}{N(\gb_i)}\right). $$ 
		
		Note that $\varphi_{R}^{-1}(C^R_{\{0\},\emptyset}) \subset C_L$ for every $L$. We approach the set $\varphi_{R}^{-1}(C^R_{\{0\},\emptyset}) $ through the sets $C_L$, by noticing that for any integer $L \geq 1$ we have the disjoint union $$C_L =  \varphi_{R}^{-1}(C^R_{\{0\},\emptyset}) \cup (C_L \setminus \varphi_{R}^{-1}(C^R_{\{0\},\emptyset})).$$  Using this, along with \Cref{eq: density and sums with T}, gives
		
		\begin{equation}
			\label{eq: big sum}
			\frac{1}{|I_N|}\sum_{a \in I_N} \one_{C_L}(T_a(\textbf{0})) = \frac{|\fr \cap I_N|}{|I_N|} + \frac{1}{|I_N|}\sum_{a \in I_N} \one_{C_L \setminus \varphi_{R}^{-1}(C^R_{\{0\},\emptyset})}(T_a(\textbf{0})).
		\end{equation}

		It follows that $$\frac{|\fr \cap I_N|}{|I_N|} \leq \frac{1}{|I_N|}\sum_{a \in I_N} \one_{C_L}(T_a(\textbf{0}))$$ holds for every $N$. Taking the limit supremum on $N$ gives  $$\bar{d}_I(\fr) \leq \prod_{i =1}^L  \left(1-\frac{|R_i|}{N(\gb_i)}\right). $$  Since this holds for every $L$, we can take the limit as $L$ goes to infinity, which gives  $$\bar{d}_I(\fr) \leq \nu_R(C^R_{\{0\},\emptyset})$$ as we wanted.
		
		Returning to \Cref{eq: big sum}, notice that $T_a(\textbf{0}) \in C_L \setminus \varphi_{R}^{-1}(C^R_{\{0\},\emptyset})$ is equivalent to $a \not \in \fr$ and $a \not \in R_j$ for all $j \leq L$. That is, $\one_{C_L \setminus \varphi_{R}^{-1}(C^R_{\{0\},\emptyset})}(T_a(\textbf{0})) = 1$ if and only if $a \in \bigcup_{i> L}R_i \setminus \bigcup_{j \leq L}R_j$, and so $$\frac{1}{|I_N|}\sum_{a \in I_N} \one_{C_L \setminus \varphi_{R}^{-1}(C^R_{\{0\},\emptyset})}(T_a(\textbf{0})) = \frac{\left|I_N \cap \bigcup_{i> L}R_i \setminus \bigcup_{j \leq L}R_j\right|}{|I_N|}.$$
		
		Hence, if on \Cref{eq: big sum} we take the limit supremum on $N$, we obtain $$ \prod_{i =1}^L  \left(1-\frac{|R_i|}{N(\gb_i)}\right) \leq \bar{d}_I(\fr) + \overline{d}_I\left(\bigcup_{i> L}R_i \setminus \bigcup_{j \leq L}R_j\right). $$ Taking the limit as $L$ goes to infinity, we get that
		\begin{equation}
			\label{eq: limsup}
			0 \leq \nu_R(C^R_{\{0\},\emptyset}) - \bar{d}_I(\fr) \leq \lim_{L \rightarrow \infty} \overline{d}_I\left(\bigcup_{i> L}R_i \setminus \bigcup_{j \leq L}R_j\right).
		\end{equation}

		On the other hand, rewriting \Cref{eq: big sum} as $$\frac{|\fr \cap I_N|}{|I_N|}  =  \frac{1}{|I_N|}\sum_{a \in I_N} \one_{C_L}(T_a(\textbf{0})) - \frac{\left|I_N \cap \bigcup_{i> L}R_i \setminus \bigcup_{j \leq L}R_j\right|}{|I_N|}$$ and taking the limit infimum on $N$ we get $$\underline{d}_I(\fr) \geq \prod_{i =1}^L  \left(1-\frac{|R_i|}{N(\gb_i)}\right) - \overline{d}_I\left(\bigcup_{i> L}R_i \setminus \bigcup_{j \leq L}R_j\right).$$ Taking the limit as $L$ goes to infinity, this gives
		\begin{equation}
			\label{eq: liminf}
			\nu_R(C^R_{\{0\},\emptyset}) \geq \underline{d}_I(\fr) \geq \nu_R(C^R_{\{0\},\emptyset}) - \lim_{L \rightarrow \infty } \overline{d}_I\left(\bigcup_{i> L}R_i \setminus \bigcup_{j \leq L}R_j\right).
		\end{equation}
		
		Hence, if $R$ has weak light tails for $I_N$, \Cref{eq: limsup} gives $\nu_R(C^R_{\{0\},\emptyset}) =   \bar{d}_I(\fr)$ and \Cref{eq: liminf} gives $\nu_R(C^R_{\{0\},\emptyset}) =   \underline{d}_I(\fr)$, so $d_I(\fr)$ exists and equals $\nu_R(C^R_{\{0\},\emptyset})$. On the other hand, if $d_I(\fr)$ exists and equals $\nu_R(C^R_{\{0\},\emptyset})$, then rewriting \Cref{eq: big sum} as $$\frac{\left|I_N \cap \bigcup_{i> L}R_i \setminus \bigcup_{j \leq L}R_j\right|}{|I_N|} = \frac{1}{|I_N|}\sum_{a \in I_N} \one_{C_L}(T_a(\textbf{0})) - \frac{|\fr \cap I_N|}{|I_N|},$$ we have that the right hand side is the sum of two sequences in $N$ with a well defined limit, so we can take the limit as $N$ goes to infinity to get $$d_I\left(\bigcup_{i> L}R_i \setminus \bigcup_{j \leq L}R_j\right) = \prod_{i =1}^L  \left(1-\frac{|R_i|}{N(\gb_i)}\right) - \nu_R(C^R_{\{0\},\emptyset}). $$ Taking $L$ to go to infinity, the right hand side goes to $0$, so $R$ will have weak light tails with respect to $I_N$.
	\end{proof}
	
	\begin{remark}
		It may happen that a sieve is such that $d_I(\fr)$ is well defined, but is different from $\nu_R(C^R_{\{0\},\emptyset})$, in which case it does not have weak light tails by  \Cref{LM: dfr smaller than nu and equivalence} (for an example look at the sieve $W'$ in \Cref{ex:Sieve with weak but not strong light tails}). It can also happen that a sieve is such that the upper density $\bar{d}_I(\fr)$ is equal to $\nu_R(C^R_{\{0\},\emptyset})$, but this still does not imply that $R$ has weak light tails (see \Cref{ex: sieve with correct upper density but no density}).
	\end{remark}
	
	Notice that in	 the proof of  \Cref{LM: dfr smaller than nu and equivalence}, we never use that $R$ is Erdős. By  \Cref{lm: finite sum finite products}, $R$ being Erdős is equivalent to $\nu_R(C^R_{\{0\},\emptyset}) >0$. Hence, if $R$ is not Erdős, we have $d_I(\fr) = 0$ for every Følner sequence $I_N$. This shows that if we didn't restrict our definition of sieve with weak light tails to Erdős sieves, we would have that every sieve $R$ such that $\sum_{\gb \in \cb_R} \vol(R_\gb) = \infty$ would have weak light tails with respect to every Følner sequence $I_N$. When $R$ is such a sieve, we have that $\nu_R = \delta_{\emptyset}$, so  $(\Omega_{R},S,\nu_R)$ is not an interesting system. For this reason, we restrict our study of sieves with weak light tails to only Erdős sieves.

	We are now in a position to show the following theorem, which is our generalization of point (1) of Sarnak's program. For the proof that $(2)$ implies $(1)$, we use ideas similar to those in the proof of Theorem 4.1 in \cite{Abdalaoui} and Theorem A (i) in \cite{Bartnicka}.
	
	\begin{theorem}
		\label{thm:Fr is generic}
		Let $R$ be an Erdős sieve. For a given Følner sequence $I_N$, the following are equivalent.
		\begin{enumerate}
			\item $\fr$ is a generic point of $(\Omega_{R},S,\nu_R)$ with respect to $I_N$.
			\item $R$ has weak light tails with respect to $I_N$.
			\item The set $\fr$ has a density with respect to $I_N$ given by $$d_I(\fr) = \nu_R(C^R_{\{0\},\emptyset}) =  \prod_{\gb \in \cbr}  \left(1-\frac{|R_\gb|}{N(\gb)}\right). $$
		\end{enumerate}
		
	\end{theorem}
	
	\begin{proof}
		
		We have shown that $(2)$ and $(3)$ are equivalent in  \Cref{LM: dfr smaller than nu and equivalence}. To show that $(1)$ implies $(3)$, notice that $C^R_{\{0\},\emptyset}$ is clopen in $\Omega_R$. Hence, if $\fr$ is generic for $I_N$, then by \Cref{lm:Sums to cardinality of sets} we have $$d_I(\fr) = \lim_{N \rightarrow \infty} \frac{1}{|I_N|}\sum_{x\in I_N} \one_{C^R_{\{0\},\emptyset}}(S_x(\fr)) = \nu_R(C^R_{\{0\},\emptyset}), $$ so $(3)$ holds. 
		
		It remains to show that $(2)$ implies $(1)$. First, notice that when $A,B$ are finite, $\one_{C^R_{A,B}}$ is a continuous function. Since $\one_{C^R_{A,B}}\one_{C^R_{X,Y}} = \one_{C^R_{A\cap X,B \cup Y}}$, the functions $\one_{C^R_{A,B}}$ generate a subalgebra of the space of continuous functions over $\Omega_R$. Clearly, this algebra separates points, since if $X \neq Y$, taking $x \in X\setminus Y$, we have $\one_{C^R_{\{x\},\emptyset}}(X) = 1$ but $\one_{C^R_{\{x\},\emptyset}}(Y) = 0$. By the Stone-Weierstrass theorem, it follows that the functions of the form $\one_{C^R_{A,B}}$ span a dense subalgebra of the set of continuous functions over $\Omega_R$. 
		
		Consequently we only need to show the result for functions of this type. In fact, using  \Cref{lm:aux for FR is generic} we see that we only have to do it for functions of the form $\one_{C^R_{A,\emptyset}}$. Arguing as in the proof of  \Cref{LM: dfr smaller than nu and equivalence} we have $\one_{C^R_{A,\emptyset}}(S_x(\fr)) = \one_{\varphi_R^{-1}(C^R_{A,\emptyset})}(T_x(\textbf{0}))$, and so we reduce the problem to showing that $$ \frac{1}{|I_N|}\sum_{x \in I_N} \one_{\varphi_R^{-1}(C^R_{A,\emptyset})}(T_x(\textbf{0})) \rightarrow  \nu_R(C^R_{A,\emptyset}) = \mmp(\varphi_R^{-1}(C^R_{A,\emptyset}))$$ as $N$ goes to infinity.
		
		The idea is again to use the fact that every point in $G_R$ is generic for the system ($G_R$,$T$,$\mmp$), as this is a minimal rotation of a compact group. 
		Yet, since $\varphi_R$ is not continuous, we must approach the set $\varphi^{-1}(C^R_{A,\emptyset})$ by clopen sets to use this fact.

		To do this, we look at the set $\varphi^{-1}(C^R_{A,\emptyset})$. If we order $\cbr$, we see that $\varphi^{-1}(C^R_{A,\emptyset})$ corresponds to those $g \in \gr$ such that $g_i+a \not \in R_i$ for every $i \in \mathbb{N}$ and $a \in A$, that is,  $$\varphi^{-1}(C^R_{A,\emptyset}) = \{g \in \gr: g_i \not \in -A + R_i \text{ for all $i$} \}.$$ Let $$C_L := \{g\in G: g_i \not \in -A + R_i \text{ for all $i$ such that } i \leq L \}.$$ It is clear that $\varphi_R^{-1}(C^R_{A,\emptyset}) \subset C_L$ for every $L$, and therefore $$C_L = \varphi_R^{-1}(C^R_{A,\emptyset}) \cup (C_L \setminus \varphi_R^{-1}(C^R_{A,\emptyset})).$$ 
		
		We now proceed as in  \Cref{LM: dfr smaller than nu and equivalence}. We have that the $C_L$ are clopen, so using  \Cref{lm: In uniquely ergodic every point generic} we know that $\textbf{0}$ is generic in $(G_R,T,\mmp)$ with respect to $I_N$, and so $$\lim_{N \rightarrow \infty} \frac{1}{|I_N|}\sum_{a \in I_N} \one_{C_L}(T_a(\textbf{0})) = \mmp(C_L) = \prod_{i=1}^L\left(1- \frac{|-A+R_i|}{N(\gb_i)}\right).$$  Therefore, by \Cref{eq:Formula for nu_R with B empty} we have  $$\lim_{L \rightarrow \infty} \lim_{N \rightarrow \infty} \frac{1}{|I_N|}\sum_{x \in I_N} \one_{C_L}(T_x(\textbf{0})) = \nu_R(C^R_{A,\emptyset}).$$ Since $$ \frac{1}{|I_N|}\sum_{x \in I_N} \one_{\varphi_R^{-1}(C^R_{A,\emptyset})}(T_x(\textbf{0})) =  \frac{1}{|I_N|}\sum_{x \in I_N} \one_{C_L}(T_x(\textbf{0})) -  \frac{1}{|I_N|}\sum_{x \in I_N} \one_{C_L \setminus \varphi_R^{-1}(C^R_{A,\emptyset})}(T_x(\textbf{0})), $$ the result will follow if we can show that $$\lim_{L \rightarrow \infty} \limsup_{N \rightarrow \infty} \frac{1}{|I_N|}\sum_{x \in I_N} \one_{C_L \setminus \varphi_R^{-1}(C^R_{A,\emptyset})}(T_x(\textbf{0})) = 0,$$
		which corresponds to showing that $$\lim_{L \rightarrow \infty}  \overline{d}_I\left(\bigcup_{i>L}(-A+R_i)\setminus\bigcup_{i \leq L}(-A+R_i)\right)= 0. $$ This is because $T_x(0) \not  \in \varphi_R^{-1}(C^R_{A,\emptyset})$ occurs if and only if there is some $i$ such that $x \in -A+R_i$, and $T_x(0) \in C_L$ is equivalent to $x \not \in -A+R_i$ for every $i \leq L$.

		We now notice that \[
		\begin{aligned}
			\left|I_N \cap \bigcup_{i>L}(-A+R_i)\setminus\bigcup_{i \leq L}(-A+R_i)\right|
			&\leq
			\sum_{a \in A}\left|I_N \cap \bigcup_{i>L}(-a+R_i)\setminus\bigcup_{i \leq L}(-A+R_i)\right| \\
			&\leq
			\sum_{a \in A}\left|I_N \cap \bigcup_{i>L}(-a+R_i)\setminus\bigcup_{i \leq L}(-a+R_i)\right| \\
			&=
			\sum_{a \in A}\left|(I_N+a) \cap \bigcup_{i>L}R_i\setminus\bigcup_{i \leq L}R_i\right|.
		\end{aligned}
		\] 
		
		Using the fact that for any three sets $X,Y,Z$ we have $|X \cap Y| \leq |X \cap Z|+ |Y \Delta Z|$, we get (taking $X = \bigcup_{i>L}R_i\setminus\bigcup_{i \leq L}R_i$, $Y = (I_N+a)$ and $Z = I_N$) that  $$\left|(I_N+a) \cap \bigcup_{i>L}R_i\setminus\bigcup_{i \leq L}R_i\right| \leq  \left|I_N \cap \bigcup_{i>L}R_i\setminus\bigcup_{i \leq L}R_i\right| + |(I_N+a)\Delta I_N|,$$ and so all we need to show is that $$	\lim_{L \rightarrow \infty} \limsup_{N \rightarrow \infty} \frac{|A|}{|I_N|}\left|I_N \cap \bigcup_{i>L}R_i\setminus\bigcup_{i \leq L}R_i\right| + \sum_{a \in A}\frac{|(I_N+a)\Delta I_N|}{|I_N|}  = 0.$$  The first term will go to $0$ since $R$ has weak light tails, while the second will go to $0$ since $I_N$ is a Følner sequence.		
	
	\end{proof}

	\subsection{Density of $R-$free numbers}
	
	The objective of this subsection is to prove a number of facts about the densities of sets of the form $\fr$. Most of the results in this section don't have a good analogue over number fields, so everything is done over $\mathbb{N}$. In particular, we will use the notation $d(A)$ for the density along the Følner sequence $I_N = [1,N]$. 
	
	We now define \index{Logarithmic Density}logarithmic density.
	
	\begin{definition}
		Given a set $A \subset \mathbb{N}$, we define the logarithmic density of $\delta(A)$ as the quantity $$\delta(A) = \lim_{N \rightarrow \infty}\frac{1}{\log(N)} \sum_{m \in A\cap [1,N]} \frac{1}{m}.$$ The upper and lower densities $\overline{\delta}$ and $\underline{\delta}$ are then defined by taking the limsup and liminf respectively. 
	\end{definition}
	
	It is well known that for any $A\subset \mathbb{N}$, the following inequalities always hold:
	$$\underline{d}(A) \leq \underline{\delta}(A) \leq \overline{\delta}(A) \leq \overline{d}(A).$$ In particular, if $d(A)$ exists, then $\delta(A)$ exists and equals $d(A)$.
	
	A powerful tool when working with $\cb-$free systems, where $\cb$  is a set of possibly non-pairwise coprime integers, is the Davenport–Erdős Theorem\index{Davenport–Erdős Theorem} (see for example the proof of Theorem 4.1 in \cite{Dymek} for a use of the theorem in this context). Given a set $\cb = \{b_1,b_2,b_3,\dots\}$, let $$\cf_{b_1,\dots,b_k} = \left(\bigcup_{i=1}^k b_i\z\right)^c.$$ We will also use the notation $\cf_\cb$ for the set obtained by sieving by every class $b_i\z$. The Davenport–Erdős theorem states the following (see \cite{Davenport}).
	
	\begin{theorem}
		Let $\cb = \{b_1,b_2,b_3,\dots\}$ be a collection of distinct integers. Then, $\delta(\cf_\cb)$ exists and we have $$\delta(\cf_\cb) = \overline{d}(\cf_\cb) = \lim_{k \rightarrow \infty}d(\cf_{b_1,\dots,b_k}).$$
	\end{theorem}

	We now show that there is no Davenport–Erdős theorem for general sieves, by providing a sieve $R$ for which $\delta(\fr)$ does not exist\footnote{In particular, this shows that $d(\fr)$ does not exist, which implies that Proposition 2.2 in \cite{Ekedahl} is incorrect.}. 
	
	\begin{proposition}
		\label{prop: sieve without density}
		There is a sieve $R$ for which $\delta(\fr)$ does not exist.\footnote{Independent from us, user ``Leeham'' in the comment section of \cite{Bloom2} used a Large Language Model to find another example of such a set. Their investigation was motivated by the fact that prior to the 12/01/2026, it was incorrectly claimed in this source that finding such a sieve was an open problem of Erdős.}
	\end{proposition}
	
	\begin{proof}

		Start by defining $S = \bigcup_{i=0}^\infty ]2^{2^{2i}},2^{2^{(2i+1)}}]$ and $I_N = [1,N]$. We have $\bar{\delta}(S) = 2/3$ and $\underline{\delta}(S) = 1/3$. Indeed, since $f(x) = 1/x$ is a non-increasing function, we can bound the Riemann sum by certain integrals which gives $$\sum_{u < m \leq v }\frac{1}{m} = \log(v)-\log(u)+c$$ for some small constant $c$. It is now easy to see that by considering the set $S$ intersected with intervals of the form $[1,2^{2^{(2k+1)}}]$, we have $$\bar{\delta}(S) = \lim_{k \rightarrow \infty} \frac{1}{2^{2k+1}\log(2)}\sum_{i \leq k} \sum_{2^{2^{2i}}<m\leq2^{2^{(2i+1)}}}\frac{1}{m} = \lim_{k \rightarrow \infty} \frac{1}{2^{2k+1}\log(2)}\sum_{i \leq k}(4^i\log(2)+c) = \frac{2}{3}. $$
		On the other hand, by considering the intersection of $S$ with sets of the form $[1,2^{2^{(2k)}}]$ we have $$\underline{\delta}(S) = \lim_{k \rightarrow \infty} \frac{1}{2^{2k}\log(2)}\sum_{i < k} \sum_{2^{2^{2i}}<m\leq2^{2^{(2i+1)}}}\frac{1}{m} = \lim_{k \rightarrow \infty} \frac{1}{4^{k}\log(2)}\sum_{i < k}(4^i\log(2)+c) = \frac{1}{3}. $$

		We now consider the sieve $R$ defined by $R_i = i\one_S(i)+p_i^4\z$, where $\one_S$ denotes the characteristic function of $S$. We want to show that $\delta(\fr)$ is not well defined.
		
		To do this, first notice that any number in $\fr$ must be in $S^c$, so we get $\fr\cap I_N \subset S^c\cap I_N$. Next, notice that for a number $x$ to be sieved out from $I_N$, either $x \in S\cap I_N$, or $x = i\one_S(i) + p_i^4k$, with some $k \geq 1$. Since $x \leq N$, we get $i \leq i\one_S(i)+p_i^4k \leq N $. Consequently, we get that $$\sum_{m \in S^c\cap I_N} \frac{1}{m}- \sum_{i \leq N}\sum_{1 \leq k \leq N/p_i^4} \frac{1}{i+kp_i^4} \leq \sum_{m \in \fr \cap I_N} \frac{1}{m}.$$ By using the fact that $\sum_{1\leq k \leq N}\frac{1}{k} \leq \log(N) +1$, we can bound the double sum by $$\sum_{i \leq N}\sum_{1 \leq k \leq N/p_i^4} \frac{1}{i+kp_i^4}  \leq \sum_{i \leq N}\frac{1}{p_i^4}\sum_{1 \leq k \leq N/p_i^4} \frac{1}{k} \leq \sum_{i \leq N} \frac{1}{p_i^4}(\log(N)+1).$$ The series $\sum_i p_i^{-4}$ can be bounded by $\zeta(4)-1$, so, using the fact that $\fr\cap I_N \subset S^c\cap I_N$, we have $$\sum_{m \in S^c\cap I_N} \frac{1}{m} - (\zeta(4)-1)(\log(N)+1) \leq \sum_{m \in \fr \cap I_N} \frac{1}{m} \leq \sum_{m \in S^c\cap I_N} \frac{1}{m}.$$
		
		We have that $\zeta(4)-1< 1/10$, so taking $N_i = 2^{2^{2i+1}}$ and $M_i = 2^{2^{2i}}$, we have that for $i$ big enough, $$\frac{1}{3}-\frac{1}{10} + \epsilon_i \leq\frac{1}{\log(N_i)}\sum_{m \in \fr \cap I_{N_i}} \frac{1}{m} \leq \frac{1}{3} + \epsilon_i,$$ and $$\frac{2}{3}-\frac{1}{10} + \epsilon_i \leq\frac{1}{\log(M_i)}\sum_{m \in \fr \cap I_{M_i}} \frac{1}{m} \leq \frac{2}{3} + \epsilon_i,$$ where $|\epsilon_i| \rightarrow 0$ as $i$ increases. For sufficiently big $i$, these two intervals will not intersect, so $\fr$ cannot have a well defined logarithmic density.  
	\end{proof}
	
	The key thing here is that our sieve $R$ was built in such a way that the set $\{\min_{x \in R_i} |x|: i \in \mathbb{N}\}$ does not have logarithmic density. But what would happen if we didn't sieve out these elements, that is, if we only sieved out those $i+kp_i^4$ with $k\geq 1$. Should this set have logarithmic density? This was first asked by Erdős in Point (26) of \cite{Erds4}, and then again in \cite{Erds2}. See also Problems number 25 and 486 in the Erdős Problems website (\cite{Bloom} and \cite{Bloom2}) for discussion on this problem. We answer the question affirmatively, in the case where we are sieving out by a sieve such that $\sum_i \frac{|R_i|}{b_i}< \infty$.

	\begin{theorem}
		\label{thm: Erdős conjecture}
		Let $\cb = \{b_1,b_2,\dots\}$ be a sequence of positive integers. For each $i$, let $R_i$ be sets of the form $$R_i = \{r_i^{(k)}: 0 \leq k \leq c(i)\} +b_i\mathbb{N}$$ where $0 \leq r_i^{(k)} < b_i$. Then, if $\sum_i |R_i|/b_i < \infty,$ we have $$d\left((\bigcup_i R_i)^c\right) = \lim_{L \rightarrow \infty}d\left((\bigcup_{i=1}^L R_i)^c\right). $$
	\end{theorem} 
	
	\begin{proof}
		Let $I_N = [0,N]$ and define $$\fr = (\bigcup_i R_i)^c.$$ 
		
		Since $$\fr^c = \bigcup_{l>L}R_l \cup \bigcup_{i=1}^LR_i,$$ the union not necessarily being disjoint, we simply have to show that $$\lim_{L \rightarrow \infty}\overline{d}_I(\bigcup_{i>L}R_i) = 0,$$ and the result will follow. For every $x\in R_i$, we have that $x\geq b_i$, consequently, for any fixed $N$, there are only finitely many $i$ such that $R_i \cap I_N \neq \emptyset$. Additionally, we have that $|R_i \cap I_N|\leq N|R_i|/b_i$, given that $|\{r_i^{(k)}+jb_i:j\in \mathbb{N}_0\}\cap I_N| \leq 1+N|R_i|/b_i$, and we are certainly removing the element $r_i^{(k)}$ for each $k$. Consequently, we have that $$\left|\bigcup_{i>L}R_i \cap I_N\right| \leq \sum_{i>L}|R_i \cap I_N| \leq N \sum_{i>L}\frac{|R_i|}{b_i}.$$ By our hypothesis, the series converges, so after dividing by $N$ taking the limit of $N$ to infinity, and then taking $L$ to infinity, this will go to $0$. Notice that $$\lim_{L \rightarrow \infty}d\left((\bigcup_{i=1}^L R_i)^c\right) $$ is always well defined, since it is the limit of a monotonic sequence.	
	\end{proof}
	
	\begin{remark}
		\label{rmk: density representative 0}
		In the case where the $b_i$ are pairwise coprime, \Cref{thm: Erdős conjecture} shows that $$d\left((\bigcup_i R_i)^c\right) = \prod_i\left(1-\frac{|R_i|}{b_i}\right). $$ Writing $R'_i = \bigcup_{1\leq k\leq c(i)} r_i^{(k)} + b_i\z$ (using the same notation as in the statement of \Cref{thm: Erdős conjecture}), we have that $R'$ is an Erdős sieve. If $S$ denotes the set of all $r_i^{(k)}$, we see that $((\bigcup_i R_i)^c \setminus S)\cap \mathbb{N}  = \frp \cap  \mathbb{N}.   $  Consequently, if $d(S) = 0$, then $d(\frp) = \prod_i\left(1-\frac{|R_i|}{b_i}\right)$, which implies that $R'$ has weak light tails for $I_N = [1,N]$ by \Cref{thm:Fr is generic}.
	\end{remark}
	
	A question that arises from  \Cref{LM: dfr smaller than nu and equivalence} is if $\overline{d}_I(\fr) = \nu_R(C^R_{\{0\},\emptyset})$ for some sieve $R$, does this imply that $d_I(\fr) = \nu_R(C^R_{\{0\},\emptyset})$? This is not the case as we show in the following example.
	
	\begin{example}
		\label{ex: sieve with correct upper density but no density}
		Let $A$ be a set with lower density $0$ but positive upper density, and let $R$ be the sieve defined by $R_i = i \one_A(i)+p_i^2\z$. Let $\mathcal{F}_R' = \mathbb{N} \setminus \bigcup_i \{i \one_A(i) +jp_i^2: j \in \mathbb{N}\}$. Writing $I_N = [1,N]$, we have shown in the proof of  \Cref{thm: Erdős conjecture} that $d_I(\mathcal{F}_R') = \nu_R(C^R_{\{0\},\emptyset})$. We have that $$|\fr \cap I_N| \geq |\mathcal{F}_R'\cap I_N| - |A \cap I_N|.$$ Taking a sequence $N_i$ so that $|A \cap I_{N_i}|/|I_{N_i}|$ tends to $0$, we get that $$\lim_i \frac{|\fr \cap I_{N_i}|}{| I_{N_i}|} \geq \lim_i \frac{|\mathcal{F}_R'\cap I_{N_i}|}{| I_{N_i}|} - \frac{|A \cap I_{N_i}|}{| I_{N_i}|} = \prod_i \left(1-\frac{1}{p_i^2}\right),$$ so $\overline{d}_I(\fr) = \nu_R(C^R_{\{0\},\emptyset})$. On the other hand, we always have that $|\fr \cap I_N| \leq |A^c \cap I_N|$, so $$\frac{|\fr \cap I_N|}{|I_N|} \leq 1- \frac{|A \cap I_N|}{|I_N|}.$$ By choosing a set $A$ such that $\overline{d}(A) >1- \prod_i \left(1-\frac{1}{p_i^2}\right)$, we get that $$\underline{d}(\fr) < \prod_i \left(1-\frac{1}{p_i^2}\right),$$ and so we cannot have that $d(\fr)$ is well defined.
	\end{example}

	\section{Topological Dynamical Systems}
	
	In this section we investigate some of the properties of the topological dynamical system $(\Omega_R,S)$. In particular, we will show generalizations of points $(2),(4)$ and $(5)$ of Sarnak's Program, as described in the introduction. As a consequence, we are able to better characterize the structure of $\fr$, the set of invariant measures of $(\Omega_{R},S)$, along with other properties of this dynamical system.
	
	We begin by showing point $(2)$. We start by setting notation for the more general context in which we will show this result. Let $\underline{s} = (s_1,s_2,\dots) $ be a sequence such that $s_i \leq |R_i^c|$ for every $i$. We follow the notation of \cite{Bartnicka}, and define the set
	\begin{equation}
		\label{eq: def of YR geq}
		Y_{R,\geq \underline{s}} :=  \{A \in \Omega_R: |A+\gb_i| \leq |R_i^c|-s_i\},
	\end{equation} where $|A+\gb_i|$ is the cardinality of this set in $\co_K/\gb_i$. Notice that taking $\underline{0} = (0,0,0,\dots)$ to be the sequence that is constantly $0$, we have $Y_{R,\geq \underline{0}} = \Omega_{R} $. 
	
	\begin{theorem}
		\label{thm: topological entropy}
		Let $R$ be any sieve and $\underline{s}$ a sequence such that $s_i \leq |R_i^c|$ for every $i \in \mathbb{N}$. Then,
		$$\htop(Y_{R,\geq \underline{s}},S) = \log(2) \prod_i \left(\frac{|R_i^c|-s_i}{N(\gb_i)}\right).$$ In particular, $$\htop(\Omega_{R},S) = \log(2) \prod_i \left(1-\frac{|R_i|}{N(\gb_i)}\right).$$
	\end{theorem}
	
	\begin{proof}
		Our proof follows closely the one used in the proof of Lemma 5.14 in \cite{Fabian}. By  \Cref{lm:Topological and patch  counting are same}, we have to show that $$\lim_{N \rightarrow \infty} \frac{\log_2(|\{E \subset B_N: E \in Y_{R,\geq \underline{s}}\}|)}{|B_N|} = \prod_i \left(\frac{|R_i^c|-s_i}{N(\gb_i)}\right).$$
		
		We start by showing that this is the correct upper bound. Fix some positive integer $M$, and let $$G_M = \prod_{i \leq M} \co_K/\gb_i.$$ Given sets $D_1, \dots, D_M$ such that $D_i \subset R_i^c$ and $|D_i| = s_i$, define for any $\delta = (\delta_1,\dots,\delta_M) \in G_M$ the sets $$U_{D_1,\dots,D_M}(\delta) = \left(\co_K\setminus \bigcup_{i\leq M} (\delta_i + (R_i \cup D_i))\right) \cap B_N.$$ Any $E \subset B_N$ such that $E \in Y_{R,\geq \underline{s}}$ must be a subset of some $U_{D_1,\dots,D_M}(\delta)$, and so we have that $$|\{E \subset B_N: E \in Y_{R,\geq \underline{s}}\}| \leq \sum_{\delta \in G_M} \sum_{D_1,\dots,D_M} |\{E \subset U_{D_1,\dots,D_M}(\delta)\}| =\sum_{\delta \in G_M} \sum_{D_1,\dots,D_M} 2^{|U_{D_1,\dots,D_M}(\delta)|}, $$ where the second sum runs over all possible choices of sets $D_i \subset R_i^c$ such that $|D_i| = s_i$.  \Cref{lm: equidistribution for Følner sequences} tells us that $$|U_{D_1,\dots,D_M}(\delta)| = |B_N|\prod_{i=1}^M\left(\frac{|R_i^c|-s_i}{N(\gb_i)}\right) + o(|B_N|),$$ but the error term depends on the set $U_{D_1,\dots,D_M}(\delta)$. Since there are only finitely many such sets, there is a constant $C_M$ dependent only on $M$ such that for any such set we have $$|U_{D_1,\dots,D_M}(\delta)| \leq |B_N|\prod_{i=1}^M\left(\frac{|R_i^c|-s_i}{N(\gb_i)}\right) + C_M f(N),$$ where $f = o(|B_N|)$. Therefore, $$ \log_2(|\{E \subset B_N: E \in Y_{R,\geq \underline{s}}\}|) \leq \log_2\left(|G_M|\prod_{i=1}^M {|R_i^c| \choose s_i}\right) + |B_N|\prod_{i=1}^M\left(\frac{|R_i^c|-s_i}{N(\gb_i)}\right) + C_M f(N),$$ which implies that $$\lim_{N \rightarrow \infty} \frac{\log_2(|\{E \subset B_N: E \in Y_{R,\geq \underline{s}}\}|)}{|B_N|} \leq \prod_{i=1}^M \left(\frac{|R_i^c|-s_i}{N(\gb_i)}\right).$$ 
		
		For the lower bound, take any choice of $A_i \subset R_i^c$ such that $|A_i| = s_i$ and define $R_i' = R_i \cup A_i$. Note that $\Omega_{R'} \subset  Y_{R,\geq \underline{s}}$. We pick according to a uniform distribution an element $t_i$ of $\co_K/\gb_i$ independently for every $i$. For any $x \in \co_K$, the probability that $x \not \in t_i+R'_i$ for some $i$ is $1-\vol(R'_i)$. Hence, for any fixed $N$, the expected number of $x \in B_N$ such that $x \not \in t_i+R'_i$ for any $i$ is $|B_N|\prod_i(1-\vol(R'_i))$. Consequently, there must be some choice of $t_i$ for which $$|B_N \setminus \bigcup_i (t_i+R'_i)| \geq |B_N|\prod_i(1-\vol(R'_i)).$$ Noticing that any subset of $\co_K \setminus\bigcup_i (t_i+R'_i)$ is $R'-$admissible, we have $$\log_2(|\{E \subset B_N: E \in \Omega_{R'}\}|) \geq \log_2(|\{E: E \subset B_N \setminus\bigcup_i (t_i+R'_i)\}|) =  |B_N \setminus \bigcup_i (t_i+R'_i)|,$$  	
		and so, using the fact that $$1-\vol(R'_i) = 1- \frac{|R_i|+s_i}{N(\gb_i)} = \frac{|R_i^c|-s_i}{N(\gb_i)}, $$ 	
		we get, by combining these equations, that
		$$\log_2(|\{E \subset B_N: E \in Y_{R,\geq \underline{s}}\}|) \geq \log_2(|\{E \subset B_N: E \in \Omega_{R'}\}|) \geq |B_N|\prod_i \left(\frac{|R_i^c|-s_i}{N(\gb_i)}\right),$$ which concludes the proof.
	\end{proof}

We now generalize point $(4)$ of Sarnak's program. In the case of $\cb-$free systems,  a large generalization of this result was given in Theorem 1.3 of \cite{Dymek2}. We show the following for sieves.

\begin{theorem}
	\label{thm: arbitrary holes iff B infinite}
	For a sieve $R$ over an étale $\q-$algebra $K$ the following holds.
	\begin{enumerate}
		\item $\emptyset \in X_R.$
		\item $(\Omega_R,S)$ is proximal.
		\item For any lattice $\Gamma \subset \co_K$, and $a \in \co_K$, $a+ \Gamma \not \subset \fr$.
		\item We have $$\inf\{\overline{d}_I(\fr): I_N \text{ is a Følner sequence in }\co_K\} = 0.$$
	\end{enumerate}
	
\end{theorem}

\begin{proof}
	
	We start by showing that $(1)$ holds, and that all remaining statements follow from this. To say that $\emptyset \in X_R$ is equivalent to showing that for every $N$ there is some $a_N$ such that $(a_N+B_N)\cap \fr$ is empty. We now show this holds since $\fr$ is obtained from removing infinitely many congruence classes of ideals which are pairwise coprime.
	
	Fix some $N$, and write $B_N = \{x_1,\dots,x_l\}$. For every $i$, take some $a_i$ such that $a_i\in R_i$. By the Chinese Remainder Theorem, there is some $x$ such that $x +x_i \equiv a_i \mod \gb_{i} $ for each $1\leq i \leq l$. This means that $x+x_i \not \in \fr$ for every $i$, that is, $(x+B_N) \cap \fr = \emptyset$. Consequently, $\emptyset \in X_R.$

	We now show that $(1)$ implies $(2)$. Because $S_a(\emptyset) = \emptyset$ for all $a$, we see that $\emptyset$ is the unique fixed point in $(\Omega_R,S)$. By  \Cref{lm: proximality auxil result}, it is enough to show that for any $A \in \Omega_R$ the set $\{a \in \co_K: d(a+A,\emptyset) < \epsilon\}$ is syndetic. 
	
	Let $R'$ be the sieve supported on $\cb_R$ and defined by $R'_\gb = (A+\gb)^c$. If $A$ is in $\Omega_R$, then we must have $ A+ \gb \neq \co_K$ for every $\gb \in \cb_R$, so $R'$ is indeed a well defined sieve. Additionally, we have that $a \not \in R'_\gb $ for every $a \in A$ and $\gb \in \cb_R$, so $A \subset \frp$. It follows that  $$d(a+A,\emptyset) \leq d(a+\frp,\emptyset),$$ and so we have  
	$$\{a \in \co_K: d(a+A,\emptyset) < \epsilon\} \supset \{a \in \co_K: d(a+\frp,\emptyset) < \epsilon\}. $$
	
	Hence, it is enough to show that the set $ \{a \in \co_K: d(a+\frp,\emptyset) < \epsilon\}$ is syndetic. To do this, notice that this follows from showing that $\{a \in \co_K: \frp \cap (a+B_N) = \emptyset \}$ is syndetic for every $N$. We have shown that $\emptyset \in X_{R'}$, so we know that for every $N$ there is some $a_N$ such that $\frp \cap (a_N+B_N) = \emptyset$. By definition of $\frp$, this means that for every $f \in a_N+ B_N$ there is some $\gb_f$ and $b_f \in R'_{\gb_f}$ such that $f$ is congruent to $b_f \mod \gb_f$. Let  $\gb$ be the product of all the $\gb_f$. For any $b\in \gb$, we have that $$f+b \equiv f \mod \gb_f,$$ and so $f+b \in R'_{\gb_f}$. It follows that $(a_N + B_N + \gb) \cap \frp = \emptyset$.
	
	Consequently, we have $a_N + \gb \subset \{a \in \co_K:  (a+B_N) \cap \frp = \emptyset \} $, so we just have to show that $a_N + \gb$ is syndetic. But taking any finite set such that $C+\gb = \co_K$, we have $a_N + \gb + C = \co_K$, so $\{a \in \co_K:  (a+B_N) \cap \frp = \emptyset \}$ must also be syndetic.  We conclude that $(\Omega_R,S)$ is proximal.

	If $(1)$ holds, then  for any $N$ there exists some $a_N$ such that $(a_N+B_N)\cap \fr = \emptyset.$ Given a lattice $\Gamma$ of $\co_K$ and $a\in \co_K$, we have that $|(a_N+B_N)\cap (a+\Gamma)| = |B_N\cap (a-a_N+\Gamma)|.$ By \Cref{lm:Counting under restriction}, we have $$|B_N\cap (a-a_N+\Gamma)| \gg_\Gamma N^{n}, $$ hence, for sufficiently large $N$ we have $(a_N+B_N)\cap (a+\Gamma) \neq \emptyset$, and so $a+\Gamma \not \subset \fr$. Point $(3)$ follows.
	
	Similarly, if $(1)$ holds, then taking a sequence $a_N$ such that $(a_N+B_N)\cap \fr= \emptyset$ for all $N$, we have that $I_N:= a_N+B_N$ is a Følner sequence, and we have $$d_I(\fr) = \lim_{N \rightarrow \infty}\frac{|I_N \cap \fr|}{|I_N|} = \lim_{N \rightarrow \infty}\frac{0}{|I_N|} = 0.$$ We conclude that point $(4)$ holds.

\end{proof}

\begin{remark}
	\label{rmk: no Erdős sieve has weak light tails for every Følner sequence}
	Since $X_R \subset \Omega_{R}$, we have by point $(2)$ that $(X_R,S)$ is always a proximal dynamical system. 
	
	Additionally, point (4) implies that for every Erdős sieve $R$ there is at least one Følner sequence $I_N$ such that $R$ does not have weak light tails for $I_N$. This is because if $R$ is Erdős, then $\nu_R(C^R_{\{0\},\emptyset}) >0$. Because we can find $I_N$ such that $d_I(\fr) = 0 < \nu_R(C^R_{\{0\},\emptyset})$,  \Cref{thm:Fr is generic} implies that $R$ does not have weak light tails with respect to $I_N$. Equivalently, this shows that there is no sieve $R$ that has weak light tails with respect to every Følner sequence $I_N$. 
\end{remark}

As a corollary of  \Cref{thm: arbitrary holes iff B infinite} we  also get the following result about the translation symmetries of $\fr$. 

\begin{corollary}
	\label{crl: translation (non)invariance of fr}
	Let $R$ be a sieve over an étale $\q-$algebra of degree $n$. If $\fr \neq \emptyset$, there are at most $n-1$ $\q-$independent $x_i$ such that $x_i + \fr = \fr$. In particular, if $R$ is a sieve over $\q$, there is no $x \neq 0$ such that $x + \fr = \fr $.
\end{corollary}

\begin{proof}
	If $x+ \fr \subset \fr$, then $2x+\fr = x + (x+\fr) \subset x + \fr \subset \fr$. By induction, we see that $x\z +\fr \subset \fr$. If there were at least $n$ $\q-$independent $x_i$ such that $x_i + \fr \subset \fr$, then we would have that $$x_1\z + \dots + x_n\z + \fr \subset \fr.$$ But due to the linear independence of the $x_i$, the set $\Gamma = x_1\z + \dots + x_n\z$ is a lattice inside $\co_K$. Taking any $x \in \fr$, we would have that $x+\Gamma \subset \fr$, contradicting point $(3)$ of  \Cref{thm: arbitrary holes iff B infinite}.
	
	The claim when $R$ is a sieve over $\q$ is clear, since if $m \neq 0$, then $m\z$ will be a lattice in $\z$.

\end{proof}
\begin{remark}
	In general, if $K$ is an étale $\q-$algebra of degree $n \geq 2$, there will be sieves such that $x_i +\fr = \fr$ for $(n-1)$ $\q-$independent $x_i$. For example, consider the sieve $R$ over the Gaussian primes $\z[i]$ (so $i = \sqrt{-1}$) supported on $\cb_R = \{p\z[i]:p \text{ is prime}, p \equiv 3 \mod 4\}$ and given by $$R_p = \{0,1,\dots, p-1\} + p\z[i] = \z+(ip)\z $$ whenever $p\z[i]\in\cb_R  $. Since these primes are inert in $\z[i]$, each of the ideals $p\z[i]$ is prime, so they are pairwise coprime. The $R-$free numbers for this set are given by $$\fr = \{a+mi:a \in \z, m \text{ is an integer divisible only by primes $p$ such that } p \not \equiv 3\mod 4\},$$ so it is clear that $1+\fr = \fr$.
\end{remark}

We conclude this section with the generalization of point $(5)$ of Sarnak's program. The proof presented is similar to the proof of Theorem E in \cite{Bartnicka}.

\begin{theorem}
	\label{thm: point 5 of Sarnak}
	The systems $(\Omega_R,S)$ and ($G_R$,$T$) have a non-trivial joining (see \Cref{def: Joining}).
\end{theorem}

\begin{proof}
	Let $J^* \subset \Omega_R \times G$ denote the set $$J^* = \left\{\left(A, g\right): A \in \Omega_R,g \in \prod_i (-A+R_i)^c  \right\},$$ and let $J$ be the closure of $J^*$. We claim that $J$ is the desired joining. First, notice that given $(A,g) \in J^*$, $(S_a\times T_a)\left(A,  g\right) = (-a+A,a+g). $ But, if $g_i \in (-A+R_i)^c$, then $$a+g_i \in a+(-A+R_i)^c = (a-A+R_i)^c = (-(-a+A)+R_i)^c.$$
	Therefore, $S \times T$ sends $J^*$ into $J^*$, which by continuity implies that $(S_a \times T_a)(J) \subset J$ for every $a$. Since the map is clearly invertible, it follows that $(S_a \times T_a)(J) = J$.
	
	Next, note that $\emptyset \in \Omega_{R}$, so $\{\emptyset\}\times G \subset J^*$, which means that $J$ has full projection into $G$. By definition, for every admissible $A$, $-A+R_i \neq \co_K$, so for every $i$, there is some $g_i \in (-A+R_i)^c$. This defines for every $A$ at least one element of $G$ such that $(A,g) \in J$, so $J$ also has a full projection onto $\Omega_R$.

	It remains to show that $J \neq \Omega_R \times G$. If $A$ is an admissible set such that $0 \in A$ and $(A,g) \in J^*$, then notice that $$g \in  \left(G \setminus \prod_{i} R_i\right).$$
	
	Assume that we have a sequence $(A_i, g_i) \in J^*$ such that the $A_i$ are converging to $\{0\}$ in $\Omega_{R}$. Then, for $i$ big enough, we must have that $0 \in A_i$, since $d(\{0\}, A) = 1$ for every admissible $A$ such that $0 \not \in A$. It follows that for big enough $i$, all the $g_i$ will be contained in $G \setminus \prod_{i} R_i$, so
	\begin{equation*}
		\left(\{0\}\times\prod_{i} R_i \right)\cap J  =\emptyset. \qedhere
	\end{equation*}
\end{proof}

\begin{remark}
	\label{rmk: OmegaR not weakly mixing}
	As pointed out in \cite{Bartnicka},  \Cref{thm: point 5 of Sarnak} implies that $(\Omega_R,S)$ is not a topologically weakly mixing system (see \Cref{def: Ergodic topological dynamical system}). This follows directly from Theorem II.3 in \cite{Furstenberg}, where it is shown that if a system is topologically weakly mixing, then it has no non-trivial joining with a minimal distal system (which ($G_R$,$T$) is, since $d(T_a(x),T_a(y)) = d(x,y)$).
\end{remark}

\section{Sieves with Light Tails}
 
 \Cref{thm:Fr is generic} shows us that we can only generalize point (1) of Sarnak's program to sieves $R$ that have weak light tails for some Følner sequence $I_N$. This points us to the necessity of studying this property. In this section, we first provide motivation, from a number theoretic point of view, why this is interesting. We then give a number of examples of sieves with strong light tails for $B_N$. Most important of these is \Cref{thm: K etale algebra light tails for B_n finite} where we show that if $R$ is any $\cb-$free system over an étale $\q-$algebra $K$, then $R$ has strong light tails for $B_N$. Finally, we investigate the relationship between the weak and strong light tails properties.

\subsection{Repeated Patterns in $\fr$}

Let $R$ be a sieve and assume that $A$ is any finite set, such that we want to know whether the pattern $A$ repeats infinitely in $\fr$. For example, taking $A=\{0,2\}$, this is the same as asking if there are infinitely many twin $R-$free numbers (as in, $x\in \fr$ such that $x+2$ is also in $\fr$). For finite $A$, if we know that $R$ has weak light tails with respect for some $I_N$, then this question reduces to knowing if $A$ is $R-$admissible or not.

If $A$ is not admissible, then there cannot be any $x \in \fr$ such that $x+ A \subset \fr$, since there will be some $\gb \in \cb_R$ such that $-A+R_\gb = \co_K$, and so taking $a \in A \cap  (R_\gb -x),$ we will always have that $a+x \not \in \fr$. On the other hand, for any finite and admissible $A$, \Cref{lm:Sums to cardinality of sets} gives $$d_I(\{x\in I_N: x+A \subset \fr\}) = \lim_{N\rightarrow\infty}\frac{1}{|I_N|}\sum_{x\in I_N}\one_{C^R_{A,\emptyset}}(S_x(\fr)).$$ If $R$ has weak light tails with respect to $I_N$, then \Cref{thm:Fr is generic} (together with \Cref{eq:Formula for nu_R with B empty}) implies that this limit equals $$\nu_R(C^R_{A,\emptyset}) = \prod_{\gb \in \cb_R} \left(1-\frac{|-A + R_\gb|}{N(\gb)}\right).$$ If $R$ is Erdős this product will always be bigger than $0$ since, by definition, if $A$ is admissible, then $|-A + R_\gb| < N(\gb)$, and if $A$ is finite then the product is positive by \Cref{lm: finite sum finite products} (using the fact that $|-A + R_\gb| \leq |A||R_\gb|$). We have just proven the following theorem.

\begin{theorem}
	\label{thm: density of x+A in fr}
	Let $R$ be an Erdős sieve with weak light tails for some Følner sequence $I_N$, and $A$ a finite admissible set of $R$. Then,
	$$d_I(\{x: x+A \subset \fr\}) = \prod_{i} \left(1-\frac{|-A+R_i|}{N(\gb_i)}\right)>0.$$ In particular, there are infinitely many $x$ such that $x+A \subset \fr$.
\end{theorem}

In the case that $A$ is infinite, $C^R_{A,\emptyset}$ is no longer clopen, and so we cannot use the weak light tails of $R$ to show that there are infinitely many $x$ such that $x+A \subset \fr$ if $\nu_R(C^R_{A,\emptyset}) >0$. Yet, we can still use  \Cref{LM: dfr smaller than nu and equivalence} to say the following.

\begin{theorem}
	\label{thm: density iff light tails} Let $R$ be a sieve and $A$ an $R-$admissible set.
	For any Følner sequence $I_N$, we have $$\overline{d}_I(\{x \in \co_K: x +A \subset \fr\}) \leq \nu_R(C^R_{A,\emptyset}).$$
	
	Additionally, $d_I(\{x: x +A \subset \fr\}) = \nu_R(C^R_{A,\emptyset}) >0$ holds if and only if the sieve $R'$ supported on $\cb_R$ and defined by 
	$$R'_\gb = -A+R_\gb$$
	is Erdős and has weak light tails for $I_N$.
\end{theorem}

\begin{proof}
	We have that $x+A \subset \fr$ is equivalent to $(x+ A)\cap R_\gb = \emptyset$ for every $\gb$, which is the same as $\{x\}\cap (-A + R_\gb) = \{x\} \cap R_\gb' = \emptyset$, so we see that  $$\{x\in \co_K: x+A \subset \fr\} = \frp. $$ 
	
	From  \Cref{LM: dfr smaller than nu and equivalence}, it is now clear that the result will follow if we show that $\nu_{R}(C^R_{A,\emptyset}) = \nu_{R'}(C^{R'}_{\{0\},\emptyset})$. From the definition of $\nu_R$, it is therefore enough to show that $\varphi_{R}^{-1}(C^R_{A,\emptyset}) = \varphi_{R'}^{-1}(C^{R'}_{\{0\},\emptyset})$ (since $\cb_R = \cb_{R'}$, we have that $G_R = G_{R'}$, so this equality makes sense). 
	
	Assume that $\varphi_R(g) \in C^R_{A,\emptyset}$. This is equivalent to $A \subset \varphi_R(g)$. From the definition of $\varphi_{R}$, we see that this is equivalent to $$\mathlarger\forall_{a\in A} \mathlarger\forall_i \hspace{2pt} a+g_i \not \in R_i \Leftrightarrow  \mathlarger\forall_i  \mathlarger\forall_{a\in A}\hspace{2pt}  g_i \not \in -a + R_i \Leftrightarrow  \mathlarger\forall_i \hspace{2pt} g_i \not  \in  -A + R_i \Leftrightarrow  \mathlarger\forall_i\hspace{2pt}  g_i \not  \in R'_i, $$ which is equivalent to $0 \in \varphi_{R'}(g)$. It follows that $\varphi_R(g) \in C^R_{A,\emptyset} $ if and only if $\varphi_{R'}(g) \in C^{R'}_{\{0\},\emptyset}$, which concludes the proof.
\end{proof}

\subsection{Sieves with Strong Light Tails}

We now investigate conditions which imply that a sieve $R$ has strong light tails for the Følner sequence $B_N$. We start by showing that this holds for all Erdős $\cb-$free systems over an étale $\q-$algebra. We need to first demonstrate that this holds for number fields, and then show the result in its full generality. In this section, we always assume that $\cb_R$ is ordered, and so we write $R_i$ for $R_{\gb_i}$.\index{Strong Light Tails ! for $\cb-$free systems in Number Fields}

\begin{lemma}
	\label{lm: light tails for B_n finite}
	Let $R$ be a sieve over a number field $K$ of degree $n$, such that there is a finite set $T$ for which $R_i \subset T + \gb_i$ for every $i$. If $\sum_i \frac{1}{N(\gb_i)} < \infty$, then $R$ is Erdős and has strong light tails for $B_N$.
\end{lemma}

\begin{proof}
	It is enough to prove the result for the sieve $R'$ defined by $R'_i = T + \gb_i$. Notice that $$\bigcup_i (T+\gb_i) = T+ \bigcup_i \gb_i.$$ If $x \in \left(T+ \bigcup_{i>L} \gb_i\right) \cap B_N$, then there are $t\in T,b_j\in\gb_j$ such that $x = t+ b_j$, $\|x\| \leq N$, which implies $\|b_j\| \leq N + \|t\|$. Hence, taking $	M	 := \max_{t\in T} \|t\|$,  it is clear that $$\left|\left(T+ \bigcup_{i>L} \gb_i\right) \cap B_N\right| \leq  |T|\left|\bigcup_{i>L} \gb_i \cap B_{N+M}\right|,$$
	so, without loss of generality, we can assume that $R$ is such that $R_i = \gb_i$.
	
	 The result now follows from Proposition 3.4. in \cite{Bartnicka}. We provide the proof, since this ends up being a simplified version of a part of the proof of \Cref{thm: K etale algebra light tails for B_n finite}, which might be easier to read if one first reads this proof. 
	 
	 	We are assuming that $\sum_i \frac{1}{N(\gb_i)} < \infty$, so $R$ is Erdős. We now show that it has strong light tails. Let $x \in B_N \setminus \{0\}$. If $x \in \gb_i$, then we must have that $ \lambda_1(\gb_i) \leq \|x\| \leq N$. Therefore, we have that $$\left|B_N \cap \bigcup_{i>L}\gb_i\right| \leq  1 + \sum_{\substack{i: \lambda_1(\gb_i) \leq N  \\ i> L}} |\{x \in  B_N \setminus \{0\} : x \in R_i \}|  , $$ which, after applying  \Cref{lm:Counting under restriction} gives 
	 $$\left|B_N \cap \bigcup_{i>L}\gb_i\right| \leq 1+  \sum_{\substack{i: \lambda_1(\gb_i) \leq N  \\ i> L}}\left( \frac{|B_N|}{N(\gb_i)} + O\left(1+\sum_{j=1}^{n-1}\frac{N^{j}}{\lambda_1(\gb_i)\cdots\lambda_j(\gb_i)}\right)  \right).$$
	 
	 We have to deal with three distinct sums separably, and show that once we divide by $|B_N|$, and take the limit of $N$ and then $L$ to infinity, these will go to $0$.  First, notice that  $$\lim_{L \rightarrow \infty} \lim_{N\rightarrow\infty} \frac{1}{|B_N|} \sum_{\substack{i: \lambda_1(\gb_i) \leq N  \\ i> L}} \frac{|B_N|}{N(\gb_i)} \leq \lim_{L \rightarrow \infty} \sum_{i>L} \frac{1}{N(\gb_i)} = 0, $$ as the series converges by hypothesis, so the first sum is dealt with.
	 
	 We next have to show that $$ \lim_{L \rightarrow \infty} \lim_{N\rightarrow\infty} \frac{1}{|B_N|} \sum_{\substack{i: \lambda_1(\gb_i) \leq N  \\ i> L}}1 = 0.$$ By  \Cref{lm: lambda_i asymp N(gb)(1/n)}, there is some $C$ dependent only on $K$ such that if $N(\gb_i)\leq C N^n$, then $\lambda_1(\gb_i) \leq N$. Therefore, the sum is bounded up to a constant multiple by $$\frac{1}{|B_N|}\sum_{i:N(\gb_i)\leq C N^n}1 $$ for some $C$ depending only on $K$. Since all the $\gb_i$ are coprime, the number of ideals in $\cb$ with norm smaller than $CN^n$ must be bounded by the number of prime ideals with norm smaller than $CN^n$. By  \Cref{thm: Prime Ideal Theorem}, this number is bounded by $cN^n/\log(N)$ for some constant $c$ depending only on $K$. Consequently, it follows that  $$\lim_{L \rightarrow \infty} \lim_{N\rightarrow\infty}  \frac{1}{|B_N|}\sum_{\substack{i: \lambda_1(\gb_i) \leq N  \\ i> L}} 1 \ll \lim_{L \rightarrow \infty} \lim_{N\rightarrow\infty} \frac{N^n}{|B_N|\log(N)} = 0, $$ as we wanted to show.
	 
	 We are left with showing that for any $1 \leq j \leq n-1 $, we have $$\lim_{L \rightarrow \infty} \lim_{N\rightarrow\infty}  \frac{1}{|B_N|}\sum_{\substack{i: \lambda_1(\gb_i) \leq N  \\ i> L}} \frac{N^j}{\lambda_1(\gb_i)\cdots\lambda_j(\gb_i)} = 0.$$
	 
	 Fix $j$. Using \Cref{eq: product of Lambdas} we have that $$\frac{N^j}{\lambda_1(\gb_i)\cdots\lambda_j(\gb_i)} \asymp_K \frac{\lambda_{j+1}(\gb_i)\cdots\lambda_n(\gb_i)N^{j}}{N(\gb_i)}. $$ By  \Cref{lm: lambda_i asymp N(gb)(1/n)}, we know that $\lambda_1(\gb_i) \asymp_K \lambda_n(\gb_i)  $, so there is some $C$ depending only on $K$, such that $\lambda_n(\gb_i) \leq C\lambda_1(\gb_i) $. Therefore, for any $i$ such that $\lambda_1(\gb_i) \leq N$, we have that $\lambda_n(\gb_i) \leq CN$. Since $\lambda_j(\gb) \leq \lambda_n(\gb)$ for every $j$, we have $$\frac{1}{|B_N|}\sum_{\substack{i: \lambda_1(\gb_i) \leq N  \\ i> L}} \frac{\lambda_{j+1}(\gb_i)\cdots\lambda_n(\gb_i)N^{j}}{N(\gb_i)} \leq \frac{1}{|B_N|}\sum_{\substack{i: \lambda_n(\gb_i) \leq CN  \\ i> L}} \frac{\lambda_n(\gb_i)^{n-j}N^{j}}{N(\gb_i)} \leq  \frac{N^{n}}{|B_N|}\sum_{\substack{i: \lambda_n(\gb_i) \leq CN  \\ i> L}} \frac{C^{n-j}}{N(\gb_i)}.   $$ The term $N^n/|B_N| $ is bounded by a constant only depending on $K$, so it follows that $$\lim_{L \rightarrow \infty} \lim_{N\rightarrow\infty}  \frac{1}{|B_N|}\sum_{\substack{i: \lambda_1(\gb_i) \leq N  \\ i> L}} \frac{\lambda_{j+1}(\gb_i)\cdots\lambda_n(\gb_i)N^{j}}{N(\gb_i)} \ll \lim_{L \rightarrow \infty} \sum_{\substack{ i> L}} \frac{1}{N(\gb_i)}  = 0.$$

	 Since all these limits go to $0$, we conclude that $$\lim_{L \rightarrow \infty} d\left(\bigcup_{i>L}R_i\right) =0,$$ as we wanted to show.

\end{proof}

\begin{remark}
	\Cref{lm: light tails for B_n finite} together with  \Cref{thm:Fr is generic} imply that, over any number field $K$, if a sieve $R$ is an Erdős $\cb-$free system, then $\fr$ has a well defined density. The fact that the elements of $\cb_R$ are pairwise coprime is fundamental in the proof. Indeed, in case this does not hold, this result will fail even for pseudosieves over $\q$. A famous example was presented by Besicovitch in \cite{Besicovitch}, of a $\cb-$free system such that for the corresponding pseudosieve $R$ the density of $\fr$ is not well defined.
\end{remark}

The proof of this result when $K$ is an étale $\q-$algebra uses similar ideas to the proof of  Proposition 3.4. in \cite{Bartnicka}. But when $K = K_1\times K_2 \times \dots \times K_l$ is as étale $\q-$algebra, we may have an infinite collection of ideals of the form $\co_{K_1}\times \gb_i$ where $\gb_i$ are nonzero proper ideals of $\co_{K_2\times\dots K_L}$. Clearly, it is not the case that $\lambda_1(\co_{K_1}\times \gb_i) \asymp \lambda_n(\co_{K_1}\times \gb_i)  $ as we had in the number field case. To deal with this, we will use the following lemma.

\begin{lemma}
	\label{lm: Moving sives across algebras preserves light tails}
	Let $K$ be an étale $\q$-algebra, and $L$ a number field. Let $R$ be a sieve over $K$, supported on a set $\cbr = \{\gb_1,\gb_2,\dots\}$. Let $F_N$ be an arbitrary Følner sequence in $\co_L$, and $ \gc_1,\gc_2,\dots$ be any infinite sequence of pairwise coprime nonzero ideals of $\co_L$.
	
	If $R$ is Erdős and has weak (respectively, strong) light tails for some Følner sequence $I_N$ in $\co_K$, the sieve $R'$ over $K\times L$ supported on $\{\gb_1\times \gc_1,\gb_2\times\gc_2,\dots\}$ and defined by $R'_i = R_i \times \co_K + \gb_i\times\gc_i$ is also Erdős, and has weak (respectively, strong) light tails over $I_N \times F_N$. 
\end{lemma}

\begin{proof} 
	It is easy to show that if $I_N$ and $F_N$ are Følner sequences, then so is $I_N \times F_N$. Because $$\vol(R'_i) = \frac{|R_i|N(\gc_i)}{N(\gb_i)N(\gc_i)} = \vol(R_i), $$ if $R$ is Erdős, so is $R'$. 
	
	We have that $$|\{(x,y) \in I_N\times F_N: (x,y) \in R'_i \text{ for some }  i>L \}| = |\{x \in I_N: x \in R_i \text{ for some }  i>L \}||F_N|  $$ with the analogous result applying if we wanted to show that $R'$ has weak light tails. It follows that if $R$ has weak or strong light tails for $I_N$, so does $R'$ for $I_N \times F_N$.
\end{proof}

\begin{remark}
	Notice that if $R$ has strong light tails for some $I_N$, and $R'$ is another sieve such that $R'_i \subset R_i$ for all $i$, then $$\overline{d}_I\left( \bigcup_{i>L}R'_i\right) \leq \overline{d}_I\left( \bigcup_{i>L}R_i\right) =0, $$ so $R'$ also has strong light tails for $I_N$. The same does not hold assuming only that $R$ has weak light tails for $I_N$, as is demonstrated by the sieves $W$ and $W'$ in \Cref{ex:Sieve with weak but not strong light tails}.
	
	Consequently, as a corollary of  \Cref{lm: Moving sives across algebras preserves light tails}, we have that if $R$ is a sieve over an étale $\q-$algebra $K$ with strong light tails for some $I_N$, and $R'$ is another sieve over some  étale $\q-$algebra $L$, then the sieve $W$ defined by $W_i = R_i\times R_i'$ will have strong light tails for any Følner sequence in $\co_{K \times L}$ of the form $I_N \times F_N$ with $F_N$ a Følner sequence in $\co_{L}$.
	
	Additionally, note that for our sequence $\gc_1,\gc_2,\gc_3,\dots$ of ideals of $\co_L$, we are free to take $\gc_i = \co_L$ for every $i$, since any ideal of $\co_L$ is coprime to $\co_L$.
\end{remark}

With this, we can now generalize  \Cref{lm: light tails for B_n finite} to the case where $K$ is an étale $\q-$algebra.  Up to now, we have always worked over some fixed $K$, so there was no ambiguity when we wrote $B_N$ (as defined in \Cref{eq: definition of BN}). In order to avoid any ambiguities, in the following proof we will use the notation $B_N^K$ for the set of elements of $\co_K$ that have norm (as defined in \Cref{eq: minkowsky norm}) smaller than or equal to $N$.\index{Strong Light Tails ! for $\cb-$free systems in étale $\q-$algebras}

\begin{theorem}
	\label{thm: K etale algebra light tails for B_n finite}
	Let $R$ be a sieve over an étale $\q-$algebra $K$ of degree $n$, such that $\cb_R =\{\gb_1,\gb_2,\dots\}$ and there is a finite set $T$ for which $R_i \subset T + \gb_i$ for every $i$. If $\sum_i \frac{1}{N(\gb_i)} < \infty$, then $R$ is Erdős and has strong light tails for $B_N$.
\end{theorem}

\begin{proof}
	By proceeding as in the proof of  \Cref{lm: light tails for B_n finite}, we can assume without loss of generality that $R_i = \gb_i$ for every $i\in \mathbb{N}$. Writing $K =K_1 \times \dots \times K_l$, we show the result by induction in $l$. The case $l=1$ corresponds to  \Cref{lm: light tails for B_n finite}. Now, assume that it holds for $l-1$. We write $K = K' \times K_l$, and $\gc_i = \ga_i\times\gb_i$ where $\ga_i$ are ideals of $\co_{K'}$ and $\gb_i$ are ideals of $\co_{K_l}$, so that $\cb_R = \{\ga_i\times\gb_i:i\in \mathbb{N}\}$. We may have $\ga_i = \co_{K'}$ or $\gb_i = \co_{K_l}$, but we can't have both simultaneously. Hence, we have the disjoint union $$\cb_R = \bigcup_{j=1}^3 \{\ga_i\times\gb_i:i\in I_j\},$$ where $I_1 = \{i: \ga_i = \co_{K'}\}, I_2 = \{i: \gb_i = \co_{K_l}\}$ and $I_3 = \{i:\ga_i \neq \co_{K'}, \gb_i \neq \co_{K_l} \}.$
	
	If any of the sets $I_j$ is finite, we don't have to consider them, since the strong light tails property is clearly invariant under changing (including removal) of any finite number of $R_i$. Hence, we see that the proof of this theorem reduces to showing that the sieves $R^{(j)}$ supported on $\cb_j := \{\gc_i: i \in I_j\}$ and defined by $R^{(j)}_i = \gc_i$ all have strong light tails with respect to $B_N^K$ for those $j$ such that $I_j$ is infinite.
	
	The fact that $\sum_i \frac{1}{N(\gc_i)} < \infty$ implies that $\sum_{i \in I_j} \frac{1}{N(\ga_i)N(\gb_i)} < \infty$ for all $j$. Assume that $I_1$ is infinite. This would imply that $\sum_{i \in I_1} \frac{1}{N(\gb_i)} < \infty $, which by  \Cref{lm: light tails for B_n finite} means that the sieve $Z^{(1)}$ supported on $\cb_{Z^{(1)}} = \{\gb_i: i \in I_1\}$ and defined by $Z^{(1)}_i := \gb_i$ over $K_l$ has strong light tails for $B_N^{K_l}$. Consequently, by  \Cref{lm: Moving sives across algebras preserves light tails}, the sieve $R^{(1)}$ will have strong light tails with respect to $B^K_N = B_N^{K'}\times B_N^{K_l}$. The same argument can be used to show that $R^{(2)}$ has strong light tails for $B^K_N$, this time using the induction hypothesis which tells us that because $\sum_{i \in I_2} \frac{1}{N(\ga_i)} < \infty $, the sieve $Z^{(2)}_i := \ga_i$ (with $i \in I_2$) over $K'$ has strong light tails for $B_N^{K'}$.
	
	Therefore, we have reduced the problem to showing that $R^{(3)}$ has strong light tails for $B_N^K$ if $I_3$ is infinite. We will use the exact same ideas already present in the proof of  \Cref{lm: light tails for B_n finite}. In order for the notation not to become overwhelming, we will just do the proof for the case $l=2$, but it should be clear that the exact same can be done for arbitrary $l$. We also assume simply that $I_3 = \mathbb{N}$, so that $R=R^{(3)}$.
	
	Let $x \in B_N^K\setminus \{0\}$. This means that writing $x = (x_1,x_2)$ with $x_1 \in K_1$ and $x_2 \in K_2$, we have $x_1 \in B_N^{K_1}$ and $x_2 \in B_N^{K_2}$. Therefore, if $x \in \ga_i \times \gb_i$, then $\lambda_1(\ga_i) \leq \|x_1\|\leq N$ and  $\lambda_1(\gb_i) \leq \|x_2\|\leq N$. Proceeding as in  \Cref{lm: light tails for B_n finite}, we get that 
	
	$$\left|B_N^K \cap \bigcup_{i>L}\ga_i \times\gb_i\right| \leq 1+  \sum_{\substack{i: \lambda_1(\ga_i) \leq N  \\ \lambda_1(\gb_i) \leq N \\ i> L}}\left( \frac{|B^K_N|}{N(\ga_i \times\gb_i)} + O\left(1+\sum_{j=1}^{n-1}\frac{N^j}{\lambda_1(\ga_i \times\gb_i)\cdots \lambda_j(\ga_i \times\gb_i)}\right)  \right).$$ 
	
	We again have to verify that after dividing by $|B_N^K|$ and taking the limit of $N$, and then $L$ to infinity, this expression will go to $0$. It is clear that $$\lim_{L \rightarrow \infty} \lim_{N\rightarrow\infty} \frac{1}{|B^K_N|} \sum_{\substack{i: \lambda_1(\ga_i) \leq N  \\ \lambda_1(\gb_i) \leq N  \\ i> L}} \frac{|B^K_N|}{N(\ga_i \times\gb_i)} \leq \lim_{L \rightarrow \infty} \sum_{i>L} \frac{1}{N(\ga_i \times\gb_i)} = 0, $$ so we now deal with the remaining terms.
	
	We have that $$  \frac{1}{|B^K_N|}\sum_{\substack{i: \lambda_1(\ga_i) \leq N \\ \lambda_1(\gb_i) \leq N  \\ i> L}} 1 \leq  \frac{1}{|B^K_N|} \sum_{\substack{i: \lambda_1(\ga_i) \leq N   \\ i> L}} 1 +\frac{1}{|B^K_N|}\sum_{\substack{i: \lambda_1(\gb_i) \leq N  \\ i> L}} 1. $$ Since the $\ga_i$ are pairwise coprime, and so are the $\gb_i$, we can bound these sums using  \Cref{thm: Prime Ideal Theorem}, and conclude that $$\lim_{L \rightarrow \infty} \lim_{N\rightarrow\infty} \frac{1}{|B^K_N|}\sum_{\substack{i: \lambda_1(\ga_i) \leq N \\ \lambda_1(\gb_i) \leq N  \\ i> L}} 1 = 0.$$
	
	Fix any $1 \leq j \leq n-1$. We are left with showing that $$\lim_{L \rightarrow \infty} \lim_{N\rightarrow\infty}  \frac{1}{|B^K_N|}\sum_{\substack{i: \lambda_1(\ga_i) \leq N  \\ \lambda_1(\gb_i) \leq N  \\ i> L}} \frac{N^j}{\lambda_1(\ga_i \times\gb_i)\cdots \lambda_j(\ga_i \times \gb_i)} = 0.$$   
	Using \Cref{eq: product of Lambdas}, we have that $$\frac{N^j}{\lambda_1(\ga_i\times \gb_i)\cdots  \lambda_j(\ga_i\times\gb_i)} \asymp_K \frac{\lambda_{j+1}(\ga_i\times\gb_i)\cdots\lambda_n(\ga_i\times\gb_i)N^{j}}{N(\ga_i)N(\gb_i)}. $$ Let $n_1$ and $n_2$ denote the degrees of $K_1$ and $K_2$ respectively such that $n = n_1+n_2$. For any $1\leq k\leq n$ we have that $\lambda_k(\ga_i\times\gb_i) \leq \lambda_{n_1}(\ga_i) + \lambda_{n_2}(\gb_i)$. By  \Cref{lm: lambda_i asymp N(gb)(1/n)}, we know that there are constants $C_1$ and $C_2$ such that $\lambda_{n_1}(\ga_i) \leq C_1\lambda_1(\ga_i) $ and $\lambda_{n_2}(\gb_i) \leq C_2\lambda_1(\gb_i) $. It follows that $$\sum_{\substack{i: \lambda_1(\ga_i) \leq N  \\ \lambda_1(\gb_i) \leq N  \\ i> L}} \frac{\lambda_{j+1}(\ga_i\times\gb_i)\cdots\lambda_n(\ga_i\times\gb_i)N^{j}}{N(\ga_i)N(\gb_i)} \leq \sum_{\substack{i: \lambda_{n_1}(\ga_i) \leq C_1N  \\ \lambda_{n_2}(\gb_i) \leq C_2N  \\ i> L}} \frac{O_j(\lambda_{n_1}(\ga_i)^{n-j} + \lambda_{n_2}(\gb_i)^{n-j})N^{j}}{N(\ga_i)N(\gb_i)}.$$ Since we are summing over $i'$s such that $\lambda_{n_1}(\ga_i) \leq  C_1N $ and $\lambda_{n_2}(\gb_i) \leq \lambda_{n_2}(\gb_i) \leq C_2N$,  the numerator must be $\ll_j (C_1+C_2)^{n-j}N^n$. Consequently, we get that $$\lim_{L \rightarrow \infty} \lim_{N\rightarrow\infty}  \frac{1}{|B^K_N|}\sum_{\substack{i: \lambda_1(\ga_i) \leq N  \\ \lambda_1(\gb_i) \leq N  \\ i> L}} \frac{N^j}{\lambda_1(\ga_i \times\gb_i)\cdots  \lambda_j(\ga_i \times \gb_i)} \ll_j \lim_{L \rightarrow \infty} \sum_{i> L} \frac{1}{N(\ga_i)N(\gb_i)} = 0,$$ which concludes our proof.
\end{proof}

\begin{example}
	\label{ex:visible lattice points}
	Let $\cb_R = \{p\z \times p\z: p \text{ prime}\}$ and define $R$ by $R_p = p\z \times p\z$. Since $\sum_p 1/p^2 < \infty $,  \Cref{thm: K etale algebra light tails for B_n finite} implies that $R$ has strong light tails for $B_N$. The set of visible lattice points corresponds to $\fr$, and we get the well known result that $$d(\fr) = \prod_{p}\left(1-\frac{1}{p^2}\right).$$
\end{example}

We now provide other examples of sieves with strong light tails for $B_N$. The idea here is that if, for some given $N$, we can control the number of $i$ such that $R_i \cap B_N \neq \emptyset$, then we may be able to show that $R$ has strong light tails for $B_N$. One way of doing this is assuming that all elements of $R_i$ have norm bigger than $N$ for all $i$ bigger than some function  $f$. Then, assuming that $f$ grows slowly enough, we are able to show that $R$ has strong light tails. Indeed, this was what we did in the proof of  \Cref{thm: K etale algebra light tails for B_n finite} using $f(i)= \lambda_1(\gb_i)$. More generally, we have the following lemma.

\begin{lemma}
	\label{lm:R has light tails when coefficients are big}
	Let $R$ be an Erdős sieve over an étale $\q-$algebra of degree $n$. If there is some function $f$ and constant $C>0$ such that $f(i) \leq C\min_{a\in R_i}\|a\|$ and $$\sum_{i:f(i) \leq N} \lambda_j(\gb_i)\cdots\lambda_n(\gb_i)\vol(R_i) = o(N^{n-j+1}),$$ for every $1 \leq j \leq n$, then $R$ has strong light tails for the Følner sequence $I_N = B_N$.
\end{lemma}

\begin{proof}

	If $x \in B_N \cap R_i$, then we get that $f(i) \leq C\|x\| \leq CN$, so we have $$\left|B_N\cap \bigcup_{i>L} R_i\right| = \sum_{\substack{i: f(i) <CN \\ i> L} }|R_i\cap B_N|.$$
	
	The set $R_i$ is the disjoint union of sets of the form $a+\gb_i$. 
	
	By  \Cref{lm:Counting under restriction} together with \Cref{eq: product of Lambdas}, we have  $$|(a+\gb_i)\cap B_N| = \frac{|B_N|}{N(\gb_i) } + O\left(\sum_{j=1}^n\frac{\lambda_j(\gb_i)\cdots\lambda_n(\gb_i)}{N(\gb_i)}N^{j-1}\right).$$ Hence, we get 
	$$\left|B_N\cap \bigcup_{i>L} R_i\right| \ll  |B_N|\sum_{i>L}\vol(R_i) + \sum_{j=1}^n N^{j-1} O\left(\sum_{\substack{i: f(i) <CN \\ i> L}}\frac{\lambda_j(\gb_i)\cdots\lambda_n(\gb_i)}{N(\gb_i)}|R_i|\right).$$

	Our hypothesis is that $R$ is so that  $$N^{j-1}\sum_{i:f(i) \leq CN}\frac{\lambda_j(\gb_i)\cdots\lambda_n(\gb_i)}{N(\gb_i)}|R_i| = N^{j-1}o(N^{n-j+1}) = o(N^n).$$
	Therefore, it is clear that $$\lim_{L \rightarrow \infty } \lim_{N \rightarrow \infty} \frac{\left|B_N\cap \bigcup_{i>L} R_i\right|}{|B_N|} = \lim_{L \rightarrow \infty}\sum_{i>L}\vol(R_i) = 0 $$ since $R$ is Erdős.
\end{proof}

Using  \Cref{lm:R has light tails when coefficients are big}, we now provide other examples of sieves  over $\q$ with strong light tails with respect to $B_N$.

\begin{example}
	Let $R$ be a sieve such that $|R_i| =1$ for every $i$, and assume that there is some function $f$ such that $f(i) \leq \min_{a\in R_i}|a|$ with $f(i) \gg i^{1+\epsilon}$ for some $\epsilon > 0$. Using \Cref{lm:R has light tails when coefficients are big}, we can easily show that $R$ has strong light tails for $B_N$.
	
	Indeed, since $|R_i| = 1$ for every $i$, all we have to show is that $\sum_{i:f(i) \leq N} 1 = o(N)$. That is, we have to show that $$|\{i: f(i) \leq N\}| = o(N).$$  But since $f(i) \gg i^{1+\epsilon}$, it is clear that $|\{i: f(i) \leq N\}| \ll N^{1/(1+\epsilon)} = o(N)$. In particular, the sieve $R$ defined by $R_i = i^2+p_i^2\z$ will have strong light tails.
\end{example}

Up to now, we have only given examples of sieves with strong light tails such that $|R_i|$ is bounded by a constant. If we assume that $|R_i|$ does not grow too fast, we can also provide examples of sieves $R$ which have strong light tails for $B_N$.

\begin{example}
	\label{ex: strong light tails if bounded by v}
	Let $v$ be the arithmetic function defined by $v(p^k) = p$, extended to all of $\mathbb{N}$ by being multiplicative. It is clear from the Prime Number Theorem, that if $\cb = \{b_i:i \in \mathbb{N}\}$ is a set of pairwise coprime numbers, we have $$|\{i: v(b_i) \leq N\}| \ll \frac{N}{\log(N)}.$$ Suppose that $R$ is a sieve such that $v(b_i) \ll \min_{a \in R_{i}} |a|$ and $\max_{i: v(b_i) \leq N} |R_{i}|= o(\log(N))$. Then $R$ has strong light tails for the Følner sequence $I_N = B_N$.
	
	Indeed, by \Cref{lm:R has light tails when coefficients are big} it is enough to show that $$\sum_{i:v(\gb_i) \leq N}|R_i| = o(N). $$ We have $$\sum_{i:v(\gb_i) \leq N}|R_i| \leq  \left(\max_{i: v(b_i) \leq N} |R_{i}| \right)|\{i: v(b_i) \leq N\}| \ll  \left(\max_{i: v(b_i) \leq N}|R_{i}| \right)\frac{N}{\log(N)}, $$ so the result follows from our hypothesis on the size of $|R_{i}|$.
	
	As an example, it follows that the sieve $R$ defined by $R_i = \{jp_i: 1 \leq j \leq \lceil\sqrt{\log(i)} \rceil\} +p_i^2\z$ has strong light tails for $B_N$.
\end{example}

Note that there are sieves for which $|R_i|$ grows linearly or even faster, that still have strong light tails, as we will show in the following theorem.

\begin{theorem}
	\label{thm: R has en equivalent sieve with stong light tails}
	Let $R$ be a sieve over a number field $K$ of degree $n$. If $R$ satisfies $$\sum_i \frac{|R_i|^{\frac{2}{n}}}{N(\gb_i)} < \infty,$$ and $$\min_{x\in R_i} \|x\| \geq \left(\frac{N(\gb_i)}{C|R_i|}\right)^{1/n}$$ for some $C$, then $R$ has strong light tails for $B_N$.
\end{theorem}

\begin{proof}

	By  \Cref{lm:R has light tails when coefficients are big}, we  just have to show that for every $j$, it holds that $$\sum_{i:\frac{N(\gb_i)}{C|R_i|} \leq N^n}  \lambda_j(\gb_i)\cdots\lambda_n(\gb_i)\vol(R_i) = o(N^{n-j+1})$$

	By  \Cref{lm: lambda_i asymp N(gb)(1/n)}, we have $\lambda_n(\gb_i) \leq cN(\gb_i)^{1/n}$ for some constant $c$, and so, if we have that $\frac{N(\gb_i)}{C|R_i|} \leq N^n$ for some $i$, then $\lambda_n(\gb_i) \leq cN(\gb_i)^{1/n} \leq C'|R_i|^{1/n}N$, where $C' :=  c C^{1/n}$. Therefore, $$\sum_{i:\frac{N(\gb_i)}{C|R_i|} \leq N^n}  \lambda_j(\gb_i)\cdots\lambda_n(\gb_i)\vol(R_i) \leq \sum_{i:\lambda_n(\gb_i) \leq C'|R_i|^{1/n}N}  \lambda_n(\gb_i)^{n-j}\frac{|R_i|}{N(\gb_i)} \leq N^{n-j}\sum_i \frac{|R_i|^{1+(n-j)/n}}{N(\gb_i)}. $$ Since by hypothesis this series converges for every $j$, the result follows.
\end{proof}

\begin{example}
	\label{ex: sieve with Ri equal pi}
	Let $R$ be a sieve defined by $R_i = \{p_i^3+j: 1 \leq j \leq p_i \} + p_i^4\z$. We have that $|R_i| = p_i$, and $\min_{x \in R_i} |x| = p_i^3$. Since $p_i^4/p_i = p_i^3$ and $$\sum_i \frac{|R_i|^2}{p_i^4} = \sum_i \frac{1}{p_i^2} < \zeta(2) < \infty,  $$ this sieve has strong light tails for $B_N$ by  \Cref{thm: R has en equivalent sieve with stong light tails}.
\end{example}

\subsection{Weak and Strong Light Tails}

We will now study a series of natural operations one can do on sieves, and investigate when they preserve weak or strong light tails. We start with the guiding example of a sieve which has weak light tails, but does not have strong light tails. 

\begin{example}
	\label{ex:Sieve with weak but not strong light tails}
	Let $\cb = \{p_i^2: i\in\mathbb{N} \}$, and  $R$ the sieve supported on $\cb$ and defined by  $$R_{i} = 1+4(i-1)+p_{i}^2\z.$$ We consider the Følner sequence $I_N = [0,N]$.
	
	If  $i \in 1+4\mathbb{N}_0$ then $i \in R_{(i-1)/4+1}$. Therefore, given any positive $N,L$, we have $$\left|I_N \cap \bigcup_{i>L}R_i \right| \geq \frac{N}{4} - 1 - L,$$ since we count every element of $(1+4\mathbb{N}_0)\cap [0,N]$, except for possibly those in $R_j$ with $j < L$. It follows that  $$\overline{d}_I\left(\bigcup_{i>L}R_i\right) \geq \frac{1}{4}$$ for all $L$, so the sieve $R$ does not have strong light tails. 
	
	On the other hand, if $x \in I_N$ is such that $x \in R_i$ for some $i> L$, but $x \not \in R_1$, then $x \geq 1 + 4(i-1) +p_{i}^2k$ with $k\geq 1$. Since $x \leq N$, it follows that if $R_i$ intersects $I_N$, then $1 +4(i-1) + p_i^2 \leq N$. In particular, it is clear that if $p_i > N$, then $R_i \cap I_N\subset \{1+4(i-1)\} \subset R_1$. 
	
	Since there are at most $1+ \frac{N}{b}$ numbers congruent to a class mod $b$ in $[0,N]$, we get that  
	$$\left|I_N\cap \bigcup_{i>L}R_i \setminus \bigcup_{j \leq L} R_j\right| \leq \sum_{\substack{i:p_i < N\\i>L}}\left(1+\frac{N}{p_i^2}\right)\leq N \sum_{i> L }\frac{1}{p_i^2} + \pi(N), $$  where $\pi(N)$ denotes the prime counting function. 
	
	Dividing by $|I_N| = N+1$ and taking $N$ to infinity, we are left with the series $\sum_{i> L} 1/p_i^2$, and since this is convergent,  when we let $L$ grow to infinity this term will go to $0$. It follows that $R$ has weak light tails.
	
	Consider now the sieve $W$, also supported on  $\{p_i^2: i \in \mathbb{N}\}$ with
	$$W_1 = 1+4\z \hspace{15pt} W_{2i} = 1+4(i-1) +p_{2i}^2\z \hspace{15pt}  W_{2i+1} = 1-4(i-1)+p_{2i+1}^2\z.$$
	The sieve $W$ is constructed so that $W_1\subset \bigcup_{i \geq 2}W_i$ but by the exact same argument as for $R$, we can show that it has weak light tails for $I_N$. Consider then the sieve $W'$ supported on $\{p_{i+1}^2: i \in \mathbb{N}\}$ and defined by $W'_i = W_{i+1}$ (alternatively, we could define $W'_1 = \emptyset$ and $W'_{i} = W_i$). Since $$\bigcup_{i \geq 1}W_i = \bigcup_{i \geq 2}W_i,$$ it follows that $\cf_W = \cf_{W'}$. But by  \Cref{thm:Fr is generic}, $W'$ does not have weak light tails for $I_N$, since $$d_I(\cf_{W'}) = d_I(\cf_{W}) = \prod_{i \geq 1} \left(1-\frac{1}{p^2}\right) < \prod_{i \geq 2} \left(1-\frac{1}{p^2}\right) = \nu_{W'}(C^{W'}_{\{0\},\emptyset}).$$
\end{example}

In this example, we took a sieve $W$ and removed an element of 
$\cb_W$ to obtain a new sieve $\cb_{W'}$. There are a number of operations we can do on a sieve $R$ by manipulating the set $\cb_R$ on which it is supported. First, the definition of weak light tails presupposes that $\cb_R$ has been ordered, and we now show that this is not problematic, since  the weak light tails property is invariant under any reordering of $\cb_R$.

\begin{lemma}
	\label{lm:Light tails invariant under }
	Let $R$ be an Erdős sieve supported on a set $\cb_R = \{\gb_1,\gb_2,\dots\}$ with weak light tails for some Følner sequence $I_N$ and $\sigma$ a permutation of $\mathbb{N}$. If $R'$ is the sieve defined by $R'_i = R_{\sigma(i)}$, then $R'$ also has weak light tails for $I_N$. 
\end{lemma}

\begin{proof}
	By  \Cref{LM: dfr smaller than nu and equivalence}, it is enough to show that $d_I(\frp) = \nu_{R'}(C^{R'}_{\{0\},\emptyset})$. It is clear that $\frp = \fr$, so since $R$ has weak light tails we get $d_I(\frp) = d_I(\fr) = \nu_{R}(C^R_{\{0\},\emptyset})$. Using \Cref{eq:Formula for nu_R with B empty}, we see that $$\nu_{R}(C^R_{\{0\},\emptyset})= \prod_i(1-\vol(R_i)) = \prod_i(1-\vol(R_{\sigma(i)})) = \nu_{R'}(C^{R'}_{\{0\},\emptyset}),$$ given that the product is absolutely convergent, and so invariant under reordering. This concludes the proof.
\end{proof}

We can also show that the same holds for the strong light tails property.

\begin{lemma}
	\label{lm:Strong light tails is invariant under permutation}
	Let $R$ be an Erdős sieve with strong light tails for some Følner sequence $I_N$ and $\sigma$ a permutation of $\mathbb{N}$. If $R'$ is the sieve defined by $R'_i = R_{\sigma(i)}$, then $R'$ also has strong light tails with respect to $I_N$. 
\end{lemma}

\begin{proof}
	Let $f:\mathbb{N} \rightarrow \mathbb{N}$ be defined by $$f(L) = \min_{i \geq L} \sigma(i).$$ This is defined so that if $i \geq L$, then $\sigma(i)\geq f(L)$. Consequently, we have $$\bigcup_{i \geq L} R'_i \subset \bigcup_{i: \sigma(i) \geq f(L)} R_{\sigma(i)} \subset \bigcup_{j \geq f(L)} R_j.$$ 
	
	Since $\sigma$ is a bijection, we have that $\lim_{L \rightarrow \infty} f(L) = \infty$, given that for any $k$, there are only $k$ solutions to $f(x) \leq k$. It follows that $$\lim_{L \rightarrow \infty}\overline{d}_I\left(\bigcup_{i \geq L} R'_i\right) \leq \lim_{L \rightarrow \infty}\overline{d}_I \left(\bigcup_{j \geq f(L)} R_j\right) = 0$$ using that $R$ has strong light tails for $I_N$.
\end{proof}

Another operation we can do on a sieve $R$ is to add congruence classes to it. There are two possibilities here. We can take a new ideal $\gb$ coprime to all of the elements of $\cb_R$, and add to $R$ some $R_\gb$, producing a new sieve $R'$ supported on $\cb_R\cup \{\gb\}$ such that $\frp = \fr \cap R_\gb^c$. Alternatively, we could also take some $\gb$ already in $\cb_R$, and some $x \not \in R_\gb$, and then consider the sieve $R'$ such that $R'_\gb = R_\gb \cup \{x+\gb\}$. In the next lemma we show that the weak light tails property is invariant under both of these operations.

\begin{lemma}
	\label{lm:adding elements presaerves light tails}
	Let $R$ be an Erdős sieve supported on the set $\cb_R = \{\gb_1, \gb_2, \dots\}$. Let $R'$ be a sieve supported on $\cb_{R'} = \{\gb_0\}\cup\cb$, with $R'_i = R_i$ for $i \geq 1$. Additionally, let $W$ be a sieve over $\cb$ such that $R_1 \subset W_1$ and $W_i=R_i$ for $i>1$. If $R$ has weak light tails for some Følner sequence $I_N$, then so do $R'$ and $W$.  
\end{lemma}

\begin{proof}
	Since $R'_i= R_i$ for all $i \geq 1$, we have that $$ \overline{d}_I\left(\bigcup_{i >L } R_i \setminus \bigcup_{1 \leq j \leq L } R_j\right) \geq \overline{d}_I\left(\bigcup_{i >L } R'_i \setminus \bigcup_{0 \leq j \leq L } R'_j\right),$$ so it follows that if $R$ has weak light tails for $I_N$, then so does $R'$.
	
	Similarly, if $W_1 \supset R_1$, then $$ \overline{d}_I\left(\bigcup_{i >L } R_i \setminus \bigcup_{1 \leq j \leq L } R_j\right) \geq \overline{d}_I\left(\bigcup_{i >L } W_i \setminus \bigcup_{1 \leq j \leq L } W_j\right),$$ so if $R$ has weak light tails for $I_N$, then so does $W$.
\end{proof}

\begin{remark}
	\label{rmk: removing classes does not preserve strong light tails}
	The strong light tails property is also invariant under these operations. Indeed, it is clear from the definition that any operation (such as addition or removal of congruence classes) that affects only finitely many $R_\gb$ will preserve strong light tails. This is distinct from weak light tails, since the removal of congruence classes from a sieve may not preserve weak light tails. This was already shown in \Cref{ex:Sieve with weak but not strong light tails}, where we defined a sieve $W$ and a sieve $W'$ obtained from $W$ by removing $W_1$. We have shown that $W$ has weak light tails for $I_N = [0,N]$, but $W'$ does not have weak light tails. 
\end{remark}

As a consequence of  \Cref{lm:adding elements presaerves light tails}, we are able to compute the density of $\fr$ intersected with some congruence class.

\begin{lemma}
	\label{lm: Fr intercepted by ideal}
	Let $\mathcal{B} = \{\gb_1, \gb_2, \cdots\}$ be an infinite set of pairwise coprime ideals. Let $R$ be an Erdős sieve over $\mathcal{B}$ with weak light tails for $I_N$. If $\gb$ is an ideal coprime to every ideal in $\mathcal{B}$, then for every $x,b \in \co_K$, $$d_I((x+\fr) \cap (b+\gb)) = \frac{d_I(\fr)}{N(\gb)}.$$
	Additionally, we have for any $i$ that  $$d_I(\fr \cap(x+\gb_i))  = \frac{d_I(\fr)}{|R_i^c|}$$ if $x + \gb_i \not \in R_i $ 
\end{lemma}

\begin{proof}
	We consider the sieve $R'$ defined by $R'_0 = (b-x+\gb)^c$ and $R'_i = R_i$ for $i > 0$. Since $R$ has weak light tails, so does $R'$ by  \Cref{lm:adding elements presaerves light tails}. We have $$\frp = \left(\bigcup_{i>0}R_i \cup R'_0\right)^c = \fr \cap (b-x+\gb),$$ and so $d_I((x+\fr)\cap (b+\gb)) = d_I(x+\frp) = d_I(\frp),$ which is equal to $d_I(\fr)N(\gb)^{-1}$ by  \Cref{thm:Fr is generic}.

	To show the second result, consider the sieve $W$ defined by $W_i = (x+\gb_i)^c$, and $W_j = R_j$ if $j \neq i$. We have $$\fr \cap (x+\gb_i) = \left( \bigcap_{j \neq i} R_j^c \cap R_i^c \right) \cap (x+\gb_i) =  \bigcap_{j \neq i} R_j^c \cap (R_i^c \cap (x+\gb_i)) = \bigcap_{j \neq i} W_j^c \cap (x+\gb_i) = \cf_W. $$  
	We have that $R_i \subset W_i $. Therefore, by  \Cref{lm:adding elements presaerves light tails}, if $R$ has weak light tails so does $W$ (we can assume that $i =1$ by  \Cref{lm:Light tails invariant under }), and we get \begin{equation*}
		d_I(\cf_W) = d_I(\fr)\left(1- \frac{|R_i|}{N(\gb_i)}\right)^{-1}\left(1- \frac{N(\gb_i)-1}{N(\gb_i)}\right) = d_I(\fr) \frac{N(\gb_i)-(N(\gb_i)-1)}{N(\gb_i)-|R_i|} = \frac{d_I(\fr)}{|R_i^c|}. \qedhere
	\end{equation*}
\end{proof}

More generally, suppose that we have a collection of sets $A_i \subset \co_K$ and we want to know the density of $\fr$ intersected with $A_i+\gb_i$. Using  \Cref{lm: Fr intercepted by ideal}, this density can be computed as follows.

\begin{lemma}
	\label{lm: Fr intercepted by multiple ideals}
	Let $R$ be an Erdős sieve with weak light tails for some Følner sequence $I_N$. Let $A_{i_j} \subset R_{i_j}^c$ be a collection of finite sets, with $1 \leq j \leq k$, such that for any distinct $a,b \in A_{i_j}$, $a + \gb_{i_j} \neq b + \gb_{i_j}$. Then
	
	$$d_I\left(\fr \cap \bigcap_{j=1}^k(A_{i_j} + \gb_{i_j}) \right) = d_I(\fr) \prod_{j=1}^{k}\frac{|A_{i_j}|}{|R_{i_j}^c|}.$$ 
\end{lemma}

\begin{proof}
	For any distinct $a,b \in A_{i_j}$ we have $a + \gb_{i_j} \neq b + \gb_{i_j}$, so $$d_I(\fr \cap (A_{i_j} + \gb_{i_j})) = \sum_{a \in A_{i_j}}d_I(\fr \cap (a+\gb_{i_j} )) = d_I(\fr)\frac{|A_{i_j}|}{|R_{i_j}^c|}$$ by  \Cref{lm: Fr intercepted by ideal}. The result follows by induction. Assume that it holds for $k-1$. Take the sieve $R'$ defined by $R'_{i} = (A_{i_j} + \gb_{i_j})^c$ if $i = i_j$ for some $1\leq j \leq k-1$ and $R'_i = R_i$ otherwise. By repeatedly using  \Cref{lm:adding elements presaerves light tails}, we have that $R'$ has weak light tails. We get $$d_I\left(\fr \cap \bigcap_{j=1}^k(A_{i_j} + \gb_{i_j}) \right) = d_I(\frp \cap (A_{i_k} + \gb_{i_k})) =  d_I(\frp)\frac{|A_{i_k}|}{|R_{i_k}^c|} = d_I(\fr)\prod_{j=1}^{k}\frac{|A_{i_j}|}{|R_{i_j}^c|},$$ the last equality following from the induction hypothesis. 
\end{proof}

\begin{example}
	Let $R$ be the squarefree sieve $R_i = p_i^2\z$. Then it is well known that $d(\fr) = \zeta(2)^{-1} = 6/\pi^2$. By  \Cref{lm: Fr intercepted by multiple ideals} it follows that the density of even squarefree numbers is $(2/3)\zeta(2)^{-1} = 4/\pi^2$ (this result was shown in \cite{Jameson} for example).
\end{example}

\begin{definition}
	\label{def: local global principle}
	We say that a sieve $R$ satisfies the \index{Local Global Property}local global principle, if after choosing any finite subset $S$ of $\cb$ and $x_\gb$ such that $x_\gb \not \in R_\gb$ for any $\gb \in S$, we can find infinitely many $y\in \fr$ such that $y \equiv x_\gb \mod \gb$. 
\end{definition}

In \cite{Fabian} (from which we take the name 'local global principle' for this property of $R$) it is pointed out that there is a close relation between light tails of $R$ and the local global principle of $R$. By   \Cref{lm: Fr intercepted by multiple ideals}, sieves with weak light tails for some $I_N$ not only satisfy the local global principle, but in fact, for any choice of $S$ and $x_\gb$, the density of $y \in \fr$ congruent to $x_\gb$ for every $\gb \in S$ is given by \begin{equation}
	\label{eq: local global density}
	d_I(\fr)\prod_{\gb \in S}\frac{1}{|R_{\gb}^c|}.
\end{equation} Yet, this is not an equivalence, as we show in the next example.

\begin{example}
	Consider the sieves $W$ and $W'$ from \Cref{ex:Sieve with weak but not strong light tails}. We know that $W$ has weak light tails for $I_N = [0,N]$, while $W'$ does not. We have that $\cb_W = \{p_i^2:i \in \mathbb{N}\}$, while $\cb_{W'} = \cb_W\setminus\{4\}$. Take any finite $S \subset \mathbb{N}\setminus\{1\}$, and choose $x_i \not \in W'_i$ for $i \in S$. We want to show that the density of those $y \in \cf_{W'}$ not congruent to $x_i \mod p_i^2$ is given by \Cref{eq: local global density}, in spite of this sieve not having weak light tails for $I_N$. 
	
	Indeed, since $\cf_W = \cf_{W'}$, and $W$ has weak light tails, we know that the density of such $y$ is given by $$d_I(\cf_W)\prod_{i \in S}\frac{1}{|W_{i}^c|} = d_I(\cf_{W'})\prod_{i \in S}\frac{1}{|(W'_{i})^c|}, $$ as we wanted to show.
\end{example}

This means that the local global property does not imply weak light tails. This was to be expected, since while the local global property does not depend on a Følner sequence $I_N$, weak light tails does (as shown by \Cref{rmk: no Erdős sieve has weak light tails for every Følner sequence}). Yet, this example shows that even a stronger form of the local global principle depending on some $I_N$, where we assume that the density (with respect to $I_N$) of $y\in \fr$ that are congruent to $x_\gb \mod \gb$, for some finitely many $x_\gb \not \in R^c_\gb$, is given by \Cref{eq: local global density}, does not imply that $R$ has weak light tails for $I_N$.

In general, the local global property is far less well behaved than the weak light tails property.  Yet, contrarily to weak light tails, the local global principle is preserved by removing ideals from $\cb_R$. Let $R$ be a sieve satisfying the local global principal, and $R^1$ the sieve supported on $\cb_R\setminus\{\gb_1\}$ and defined by $R^1_i = R_{i+1}$. Then $R^1$ is supported over the set $\cb_R \setminus \{\gb_1\}$, and because we are sieving out less, we have $\fr \subset \mathcal{F}_{R^1}$. It is then clear that $R^1$ must also satisfy the local global principle. Any finite $S \subset\cb_R \setminus \{\gb_1\} $ is a subset of $\cb_R$, so given any number of congruence classes modulo $\gb$ with $\gb$ in $S$, we can find a solution in $\fr$ by the local global principle of $R$, which will also be in $\mathcal{F}_{R^1}$.

Contrary to weak light tails, the local global principle for sieves is not preserved by adding ideals (which we have shown holds for weak light tails in  \Cref{lm:adding elements presaerves light tails}) as the following example shows. 

\begin{example}
	
	Let $R$ be the sieve defined by $R_1 = \{0,1\}+4\z$, $R_{2i} = 1+4i + p_{2i}^2\z$ and  $R_{2i+1} = 1-4i + p_{2i}^2\z$. Proceeding as we have done in \Cref{ex:Sieve with weak but not strong light tails}, we can show that $R$ has weak light tails for $B_N$. Let $R'$ be the sieve supported on $\{p_i: i \geq 2\}$ and defined by $R'_i = R_i$ if $i \geq 2$. The sieve $R'$ satisfies the local global principle because $R$ does. Yet, the sieve $R''$ defined by $R''_1 = 0+4\z$ and $R''_i = R_i$ for $i \geq 2$ will not satisfy the local global principle, since the only element of $1+4\z$ in $\cf_{R''}$ is $1$, even though $(1+4\z) \cap R''_1 = \emptyset$. Since $R''$ was obtained from $R'$ by adding to it the congruence class $0+4\z$, it follows that the local global principle is not preserved by adding congruence classes. Additionally, since $R''$ can be obtained from $R$ by removing the congruence class $1+4\z$ from $R_1$, it follows that the local global principle is also not preserved by removing classes.
	
\end{example}

		Below we summarize how the different operations preserve (or don't) each property.
		
		\begin{table}[h]
			\label{table: properties}
			\begin{tabular}{l|l|l|l}
				& Local global property & Weak light tails for $I_N$ & Strong light tails for $I_N$ \\ \hline
				Adding congruence classes   & does not preserve     & preserves                  & preserves                    \\
				Removing congruence classes & does not preserve     & does not preserve          & preserves                    \\
				Removing ideals             & preserves             & does not preserve          & preserves                   
			\end{tabular}
			\caption{How different operations change the properties of an Erdős sieve $R$.}
		\end{table}
		
		As can be seen, the strong light tails property is much more stable. Now, say that we have a sieve $R$ with weak light tails for some $I_N$, such that the weak light tails are preserved after repeatedly removing congruence classes. It is natural to ask if this sieve must have strong light tails. We now show that this is the case.

		\begin{theorem}
			\label{thm: R Erdős weak tails to strong tails}
			Let $R$ be an Erdős sieve and $I_N$ any Følner sequence. We have that $R$ has strong light tails for $I_N$ if and only if for every $L$, the sieve $R^L$ supported on $\cb_{R^L} = \{\gb_{L+1},\gb_{L+2},\dots \}$ and defined by $R^L_i = R_{L+i}$ has weak light tails for $I_N$.
		\end{theorem}
		
		\begin{proof}

			From the definition of strong light tails, it is clear that this property is invariant, under any change of a finite number of the $R_i$. Therefore, if $R$ has strong light tails, so does every $R^L$, which implies that every $R^L$ will have weak light tails.
			
			We now show the opposite implication. Since $R$ is Erdős, so is every $R^L$. By definition,  $\mathcal{F}_{R^{L}}$ is equal to $ \left(\bigcup_{j>L} R_j\right)^c$, so $$d_I\left(\bigcup_{j>L}R_j\right) = d_I(\mathcal{F}_{R^{L}}^c) = 1-d_I(\mathcal{F}_{R^{L}}) = 1-\prod_{j>L}\left(1-\frac{|R_j|}{N(\gb_j)}\right),$$ where in the last step we use the fact that $R^{L}$ has weak light tails and apply  \Cref{thm:Fr is generic}.  Since $R$ is Erdős, the product goes to $1$ as $L$ goes to infinity, which concludes the proof.
		\end{proof}
		
		Notice that if we replaced our hypothesis that $R^L$ has weak light tails for every $L$, by the fact that there is a sequence $L_i$ such that $R^{L_i}$ has weak light tails for every $i$, the result would still hold. Note also that in order for this result to hold, we must be able to remove any number of ideals from $R$, with it retaining the weak light tails property. It is not true that if $R$ has weak light tails for $B_N$, and for every $\gb \in \cb$, the sieve $W^{(\gb)}$ supported on $\cb_R \setminus \{\gb\}$ by $W^{(\gb)}_\ga = R_\ga$ also has weak light tails, then $R$ has strong light tails. We provide an example.
		
		\begin{example}
			Let $R$ be the sieve defined by $R_1 = 2\z$, $R_2 = 3\z$ with $$R_{2i} = 6i + p_i^2\z \hspace{5pt} \text{ and } \hspace{5pt} R_{2i+1} = -6i + p_i^2\z$$ for any $i \geq 2$. It is not hard to show that $R$ has weak light tails by proceeding as in \Cref{ex:Sieve with weak but not strong light tails}. Additionally, any sieve obtained from $R$ by removing exactly one ideal will also have weak light tails. This is because if $R'$ is a sieve obtained from $R$ by removing exactly one ideal, then there is still some $i$ such that $ 6\z \subset R'_i$ given that $6\z$ is a subset of both $R_1$ and $R_2$.
			
			Yet, it is clear that $R$ does not have strong light tails for $B_N$, as we have that $$\left|B_N \cap \bigcup_{i>L}R_i\right| > \frac{N}{3} - L-2.$$
		\end{example}

		\Cref{thm: R Erdős weak tails to strong tails} points to us the necessity of studying which sieves with weak light tails remain with weak light tails after removing ideals. We now study such properties. Our guiding example is the sieve $R$ given in \Cref{ex:Sieve with weak but not strong light tails}, which  has weak but not strong light tails. The key feature of this sieve, is that one of the congruence classes in $R$ is sieved out by the other classes. We want to show that when this does not happen for any class then $R$ will have strong light tails. We start with the following lemma.
		
		\begin{lemma}
			\label{lm: removing ideals}
			Let $R$ and $R'$ be Erdős sieves both supported on the set $\cb = \{\gb_0,\gb_1,\gb_2,\dots\}$ such that $R'_0  = R_0 \cup  (x+\gb_0)$ for some $x \not \in R_0$, and $R'_i = R_i$ for every $i>0$ ($R_0$ is allowed to be the empty set here). If $R'$  has weak light tails for $I_N$ as well as
			\begin{equation}	
				\label{eq: for lemma removing ideals}
				\lim_{L \rightarrow \infty } \overline{d}_I\left((x+\gb_0) \cap \bigcup_{i > L}R_i \setminus \bigcup_{1 \leq j \leq L} R_j\right) =0,
			\end{equation}	
			then $R$ has weak light tails for $I_N$.
		\end{lemma}

		\begin{proof}
			We start by writing $X_L = R_0^c \cap R_1^c\cap \dots \cap R_L^c$ with $L \geq 1$. Then, the light tails condition for $R$ can be written as $$\lim_{L \rightarrow \infty} \limsup_{N \rightarrow \infty} \frac{|\fr^c \cap X_L \cap I_N|}{|I_N|} = 0.$$ We decompose $\fr^c$ into those elements which are in $x+\gb_0$, and those which are not, to obtain $$|\fr^c \cap X_L \cap I_N| = |(\fr^c\cap (x+\gb_0)) \cap X_L \cap I_N| + |(\fr^c\cap (x+\gb_0)^c) \cap X_L \cap I_N|,$$ and since $(x+ \gb_0)^c \cap R_0^c = (R_{0}')^c$, we get that $$(\fr^c\cap (x+\gb_0)^c) \cap X_L \cap I_N = I_N \cap \bigcup_{i > L}R'_i \setminus \bigcup_{0 \leq j \leq L} R'_j.$$
			
			It follows that because $R'$ has weak light tails, we have  $$\lim_{L \rightarrow \infty} \limsup_{N \rightarrow \infty} \frac{ |(\fr^c\cap (x+\gb_0)^c) \cap X_L \cap I_N|}{|I_N|} = 0. $$On the other hand, because $(x+\gb_0)\cap (R_0)^c = x+\gb_0$, we have that $$(x+\gb_0)\cap X_L = (x+\gb_0) \setminus  \bigcup_{1 \leq j \leq L} R_j.$$ so we have $$\lim_{L \rightarrow \infty} \limsup_{N \rightarrow \infty} \frac{  |(\fr^c\cap (x+\gb_0)) \cap X_L \cap I_N|}{|I_N|} =  \lim_{L \rightarrow \infty } \overline{d}_I\left((x+\gb_0) \cap \bigcup_{i > L}R_i \setminus \bigcup_{1 \leq j \leq L} R_j\right) =0, $$ where the second equality is by hypothesis. The result follows.
		\end{proof}
		
		With this lemma, it is now easy to show the following theorem, which formalizes the idea that if $R$ is a sieve with weak light tails, and no congruence class is being sieved out by the others, then $R$ has strong light tails.
		
		\begin{theorem}
			\label{thm: R has strong tails if no congruence is sieved out}
			If $R$ is an Erdős sieve with weak light tails for $I_N$, such that for every $j\in \mathbb{N}$ and $x_j \in R_j$, we have 
			\begin{equation}
				\label{eq: for R has strong tails if no congruence is sieved out}
				\lim_{L \rightarrow \infty } \overline{d}_I\left((x_j+\gb_j) \cap \bigcup_{i > L}R_i\right) = 0,
			\end{equation}
			then $R$ has strong light tails for $I_N$.
		\end{theorem}
		
		\begin{proof}
			
			We see that if $R$ is an Erdős sieve such that, for every $x_j \in R_j$, \Cref{eq: for R has strong tails if no congruence is sieved out} holds, then the same will happen for a sieve $R'$ obtained from $R$ by removing an ideal. If $R$ satisfies \Cref{eq: for R has strong tails if no congruence is sieved out}, then \Cref{eq: for lemma removing ideals} also holds for $R$, so if $R$ is a sieve with weak light tails, and $R'$ is the sieve obtained by removing $x_j + \gb_j$ from $R$, \Cref{lm: removing ideals} shows that $R'$ will have weak light tails. The result then follows directly from  \Cref{thm: R Erdős weak tails to strong tails}.
		\end{proof}

		\Cref{thm: R has strong tails if no congruence is sieved out} shows that if $R$ has weak light tails for $I_N$, but not strong light tails,  then there must be some $j \in \mathbb{N}$ and $x_j \in R_j$ such that $(x_j + \gb_j)$ has a subset of positive $I_N-$density that is sieved out by the $R_i$ with $i \neq j$. If $R$ is such a sieve, then by taking some $y$ such that $x_j \not \in y+ R_j$, we would expect the sieve $R'$ defined by $R'_j = y+R_j$ and $R'_i = R_i$ for $i \neq j$ not to have weak light tails for $I_N$.
		
		It is natural to ask for the inverse of this, that is, if we have a sieve $R$ with weak light tails for $I_N$ such that translations of $R_i$ preserve the weak light tails condition, must it have strong light tails? Indeed, this is the case, as we now show.

		\begin{theorem}
			\label{thm: R has strong light tails if finite translations}
			Let $R$ be an Erdős sieve with weak light tails with respect to some Følner sequence $I_N$. Suppose that for any finite set $A \subset \mathbb{N}$ and choice of $x_i \in \co_K$ for $i \in A$, the sieve $R'$ defined by $R'_i = x_i + R_i$ with $i \in A $, and $R'_i = R_i$ otherwise, also has weak light tails. Then $R$ has strong light tails for $I_N$.
		\end{theorem}
		
		\begin{proof}

			We want to show that for any $M$, the sieve $R^M$ defined by $R^M_i = R_{M+i}$ has weak light tails, and then apply  \Cref{thm: R Erdős weak tails to strong tails}. This is equivalent to showing that $$\lim_{L \rightarrow \infty}\overline{d}_I\left(\bigcup_{i > L} R_i \setminus \bigcup_{M < j \leq L} R_j \right) = 0.$$ Our hypothesis implies that for any choice of $x_j \in \co_K$ with $1 \leq j \leq M$ we have 
			$$\lim_{L \rightarrow \infty}\overline{d}_I\left(\bigcup_{i > L} R_i \setminus\left(\bigcup_{1 \leq j \leq M} (x_j + R_j) \cup  \bigcup_{M < j \leq L} R_j  \right) \right) = 0.$$
			
			Let $G_M$ be the finite set defined by $$G_M = \prod_{j=1}^M \co_K/\gb_j.$$ Then, from the equality $$\bigcup_{i > L} R_i \setminus\left( \bigcup_{1 \leq j \leq M} (x_j + R_j) \cup  \bigcup_{M < j \leq L} R_j  \right) = \left(\bigcup_{i > L} R_i \cap \bigcap_{M < j \leq L} R_j^c  \right)\cap \bigcap_{1 \leq j \leq M} (x_j + R_j)^c ,$$ it is clear that each of these sets is contained in $ \bigcup_{i>L}R_i \setminus \bigcup_{M < j \leq L} R_j$, and that the result will follow if we show that $$\bigcup_{(x_1,\dots,x_m)\in G_M}\left( \bigcup_{1 \leq j \leq M} (x_j + R_j)^c \right) = \co_K.$$ 
			
			To show this, let $y_j$ be a collection of elements of $\co_K$ chosen so that $y_j \in R_j^c$ for every $1 \leq j \leq M$. Let $y$ be any arbitrary element of $\co_K$ and pick the unique $x\in G_M$ such that $x_j \equiv y - y_j \mod \gb_j$ for every $j$. Then $y \in x_j +y_j + \gb_j \subset x_j+ R_j^c$ for every $j$, which shows that $y$ is in $\bigcap_{1 \leq j \leq M} (x_j + R_j)^c$. This concludes the proof.
		\end{proof}
		
		\begin{remark}
			Since any change in finitely many $R_i$ does not affect the strong light tails property  the result in  \Cref{thm: R has strong light tails if finite translations} is an equivalence. 
			
			Additionally, this shows that if $R$ has weak light tails for some $I_N$, but not strong light tails, then there is a finite set $A \subset \mathbb{N}$ and a collection of $x_i \in \co_K$ for $i \in A$ such that the sieve $R'$ defined by $R'_i = x+ R_i$ if $i \in A$ (and $R'_i=R_i$ otherwise) does not have weak light tails for $I_N$.
			
		\end{remark}

			\bigskip
			\footnotesize
			\noindent
			Francisco Ara\'{u}jo\\
			\textsc{Institute of Mathematics, Paderborn University, Warburger Str. 100, 33098 Paderborn, Germany}\par\nopagebreak
			\noindent
			\textit{E-mail address:} \texttt{faraujo@math.uni-paderborn.de}

			\end{document}